\documentclass[reqno, a4paper]{amsart}
\usepackage{amsmath, amssymb, amsthm, enumitem}
\usepackage{placeins}
\usepackage{colonequals}

\pdfoutput=1
\usepackage{color}

\usepackage{hyperref}
\usepackage{txfonts}

\usepackage{tikz}
\usepackage{caption}
\usepackage{accents}

\newtheorem{theorem}{Theorem}

\newtheorem{corollary}[theorem]{Corollary}
\newtheorem{definition}[theorem]{Definition}
\newtheorem{lemma}[theorem]{Lemma}
\newtheorem{proposition}[theorem]{Proposition}

{Important Convention}

\theoremstyle{remark}
\newtheorem{remark}[theorem]{Remark}
\newtheorem{example}[theorem]{Example}

 \renewcommand{\phi}{\varphi}

\newcommand{\N}{\mathbb{N}}
\newcommand{\Q}{\mathbb{Q}}
\newcommand{\R}{\mathbb{R}}

\DeclareMathOperator{\supp}{supp}
\newcommand{\bes}{\begin{subequations}}
\newcommand{\ees}{\end{subequations}}
\newcommand{\eea}{\end{eqnarray}}

\usepackage{bbm}

\renewcommand{\epsilon}{\varepsilon}

\DeclareMathOperator{\proj}{proj}

\newcommand{\fourIdx}[5]{%
\setbox1=\hbox{\ensuremath{^{#1}}}%
 \setbox2=\hbox{\ensuremath{_{#2}}}%
 \setbox5=\hbox{\ensuremath{#5}}%
 \hspace{\ifnum\wd1>\wd2\wd1\else\wd2\fi}%
 \ensuremath{\copy5^{\hspace{-\wd1}\hspace{-\wd5}#1\hspace{\wd5}#3}%
 _{\hspace{-\wd2}\hspace{-\wd5}#2\hspace{\wd5}#4}%
 }}

\numberwithin{equation}{section}
\numberwithin{theorem}{section}

\renewcommand{\subset}{\subseteq}
\renewcommand{\supset}{\supseteq}

\usepackage{subcaption}
\usepackage{graphicx}

\renewcommand{\mathrm}{}

\makeatletter
\newcommand{\mylabel}[2]{#2\def\@currentlabel{#2}\label{#1}}
\makeatother
\begin{document}

\title{Stability of martingale optimal transport and weak optimal transport}
\author{J. Backhoff-Veraguas}
\author{G. Pammer}

\begin{abstract}

Under mild regularity assumptions, the transport problem is stable in
the following sense: if a sequence of optimal transport plans $\pi^1,
\pi^2, \ldots$  converges weakly to a transport plan $\pi$, then $\pi$ is
also optimal (between its marginals).

Alfonsi, Corbetta and Jourdain \cite{AlCoJo17b}
asked whether the same property is true for the martingale transport
problem. This question seems particularly pressing since
martingale transport is motivated by robust finance where data is naturally
 noisy.  On a technical level, stability in the martingale case
appears more intricate than for classical transport since martingale optimal
transport plans are not characterized by a `monotonicity'-property of their supports.

In this paper we give a positive answer and establish stability of the
martingale transport
problem. As a particular case, this recovers the stability of the left curtain coupling established by Juillet \cite{Ju16}. An important auxiliary tool is an unconventional topology
which takes the temporal structure of
martingales into account. Our techniques also apply to the the weak
transport problem introduced by Gozlan, Roberto, Samson and Tetali.

 \medskip

\noindent \emph{Keywords: stability, martingale transport, weak transport, causal transport, weak
adapted topology, robust finance.}
\end{abstract}
\maketitle

\section{Introduction and main results}

Let $X$ and $Y$ be Polish spaces and consider a continuous function $c\colon X\times Y\rightarrow [0,\infty)$. Given probability measures $\mu \in \mathcal P(X)$ and $\nu\in\mathcal P(Y)$, the classical transport problem is 
\begin{equation}\tag{OT}\label{OT}
	\inf_{\pi \in \Pi(\mu,\nu)} \int_{X\times Y} c(x,y) \, \pi(dx,dy),
\end{equation}
where $\Pi(\mu,\nu)$ denotes the set of couplings with $X$-marginal $\mu$ and $Y$-marginal $\nu$. A classical result in optimal transport asserts   that $\pi\in \Pi(\mu,\nu)$ is optimal for \eqref{OT} iff its support $\supp
\pi$ is $c$-{cyclically} monotone \cite{Vi03,Vi09}. 
 One useful consequence of this characterization of optimality is the stability of \eqref{OT} with respect to the marginals $\mu,\nu$ as well as the cost function $c$. Indeed, the link between monotonicity and stability becomes apparent once one realizes that the notion of monotonicity is itself stable.

In this article we consider the martingale optimal transport problem from the point of view of monotonicity and stability.
In fact, since this problem is an instance of a weak optimal transport problem {(where the cost is singular, i.e., also takes the value $+\infty$)}, we will likewise study the latter class of problems {for regular cost functions} from this viewpoint.

\subsection{Stability of martingale optimal transport}

The martingale optimal transport problem is a variant of \eqref{OT} stemming from robust mathematical finance (cf.\ \cite{HoNe12, BeHePe12, TaTo13, GaHeTo13, DoSo12, CaLaMa14,BeCoHu14, BaBeHuKa17,JoMa18,Ju16,ObSi17,DeTo17,HuTr17,GhKiLi16b,GuLoWa19} among many others). In order to define this problem, we take $X=Y=\R$, suppose that $\mu,\nu$ have finite first moments, and introduce the set $\Pi_M(\mu,\nu)$ of martingale couplings with marginals $\mu, \nu$. To be precise, a transport plan $\pi$ is a martingale coupling iff
\[ \int_{\R} y \, \pi_x(dy) = x\quad \mu\text{-a.s.,}\]
where $\{\pi_x\}_{x \in \R}$ denotes a regular disintegration of the second coordinate given the first one. By a famous result of Strassen, the set $\Pi_M(\mu,\nu)$ is non-empty iff $\mu$ is smaller than $\nu$ in convex order. The martingale optimal transport problem\footnote{The multidimensional version of the martingale transport problem is defined analogously, although the mathematical finance application is less clear.} is given by
\begin{equation}\tag{MOT}\label{MOT}
	\inf_{\pi\in\Pi_M(\mu,\nu)} \int_{\R\times\R} c(x,y) \, \pi(dx,dy),
\end{equation}
{under the convention that the infimum over the empty set is $+\infty$.}

The main result of the article is the stability of {\eqref{MOT}}. 
This gives a positive answer  to the question posed by Alfonsi, Corbetta and Jourdain in \cite[Section 5.3]{AlCoJo17b} in the case $d=1$.

{Let $r \geq 1$.}
 We denote by $\mathcal P_{r}(\R)$ the set of probability measures with finite {$r$th-moments} and by $\mathcal W_{r}$ the topology of $r$-Wasserstein convergence on  $\mathcal P_r(\R)$, cf.\ \cite{Vi03}.

\begin{theorem}[MOT Stability]\label{Thm stab mot intro}
	Let $c,c_k \colon \R \times \R \rightarrow [0, \infty)$, $k\in\N$, be  continuous cost functions such that $c_k$ converges uniformly to $c$. Let {$\{\mu_k\}_{k \in \N},\{\nu_k\}_{k \in \N}$ be sequences in} $\mathcal P_1(\R)$	converging in $\mathcal W_1$ to $\mu$ and $\nu$, {respectively}. For each $k \in \N$ let $\pi^k\in \Pi_M(\mu_k,\nu_k)$ be {an optimizer of \eqref{MOT} with cost $c_k$ between the marginals $\mu_k$ and $\nu_k$.}
	If $c(x,y) \leq a(x) + b(y)$ with $a \in L^1(\mu)$, $b\in L^1(\nu)$, and
	\[
		{\limsup_{k \to \infty} \int_{\R\times\R} c_k(x,y) \,\pi^k(dx,dy)} < \infty,
	\]
	then any {weak} accumulation point of $\{\pi^k\}_{k\in\N}$  is an optimizer of \eqref{MOT} for the cost function $c$. In particular if the latter has a unique optimizer $\pi$, then $\pi^k\to \pi$ weakly.
\end{theorem}

\begin{corollary}\label{Cor stab mot intro}
Let $c,c_k \colon \R \times \R \rightarrow [0, \infty)$, $k\in\N$, be  continuous cost functions such that $c_k$ converges uniformly to $c$. Let {$\{\mu_k\}_{k \in \N},\{\nu_k\}_{k \in \N}$ be sequences in} $\mathcal P_r(\R)$	converging in $\mathcal W_r$ to $\mu$ and $\nu$, {respectively,  and $\mu_k$ is smaller in convex order than $\nu_k$, $k \in \N$.}
Suppose that
\[ c(x,y)\leq { K\left(1+|x|^r+|y|^r\right)},\text{ for some }{K>0.}\]
Then we have
\[
	\lim_{k\to\infty}\,\inf_{\pi\in\Pi_M(\mu_k,\nu_k)} \int_{\R\times\R} c_k(x,y)\,\pi(dx,dy)=\inf_{\pi\in\Pi_M(\mu,\nu)} \int_{\R\times\R} c(x,y)\,\pi(dx,dy).
\]
\end{corollary}
We remark that Juillet has obtained in \cite{Ju16} the stability of the left-curtain coupling and hence stability for martingale transport for specific costs. These results are recovered as particular cases of our main result.   

Guo and Ob{\l}{\'o}j in \cite{GuOb17} introduce and study the convergence of a computational method for martingale transport where the marginals are discretely approximated and the martingale constraint is allowed to fail with a vanishing error.\footnote{Note that the updated version \cite{GuOb19} (listed on arxiv.org on April 8th 2019) shows in Proposition 4.7  that the optimal value of \ref{MOT} is continuous w.r.t. $(\mu, \nu)\in \mathcal P_2(\R) \times \mathcal P_2(\R)$ provided that $\mathcal P_2(\R) \times \mathcal P_2(\R)$ is equipped with $\mathcal W_2$-convergence and $c$ is assumed to be Lipschitz continuous. 

The authors of the present article decided to post this article on arxiv.org 	concurrently to emphasize the independence of our work. We also note that the main focus of  \cite{GuOb17}/\cite{GuOb19} lies on the numerics of martingale transport. In contrast to Theorem \ref{Thm stab mot intro} the  proof of \cite[Proposition 4.7]{GuOb19} is based on the dual problem rather than stability of the optimal couplings.
}
In independent work Wiesel \cite{Wi19} proved stability of the value function of the martingale optimal transport problem in $1$ dimension by estimating the distance between an arbitrary coupling to its projection w.r.t.\ the adapted Wasserstein distance. For further details on the adapted Wasserstein distance we refer to \cite{BaBaBeEd19}.

\subsection{Stability of optimal weak transport}

Gozlan, Roberto, Samson and Tetali \cite{GoRoSaTe17} proposed the following non-linear generalization of \eqref{OT}: Given a cost function $C:X\times \mathcal P(Y)\to\mathbb R$ the optimal weak transport problem is
\begin{equation}\tag{WOT}\label{WOT}
	\inf_{\pi\in\Pi(\mu,\nu)} \int_X C(x,\pi_x) \, \mu(dx).
\end{equation}
 Observe that one may consider cost functions of the form
\begin{equation}
	\label{MOT subset WOT}
	C_M(x,p) := \begin{cases} \int_{\R^d} c(x,y) \, p(dy) & \int_{\R^d} y \, p(dy) = x,\\ +\infty &\text{else,} \end{cases}
\end{equation}
and in this way \eqref{MOT} is a special case of \eqref{WOT}. 

While the original motivation for \eqref{WOT} mainly stems from applications to geometric inequalities (cf.\ Marton \cite{Ma96concentration, Ma96contracting} and Talagrand \cite{Ta95, Ta96}), weak transport problems appear also in a number of further topics, including martingale transport \cite{AlCoJo17, AlBoCh18, BeJu17, BaBeHuKa17,BaBePa18}, the causal transport problem \cite{BaBeLiZa16,AcBaZa16}, and  stability in mathematical finance \cite{BaBaBeEd19}. In fact, recently some works have considered non-linear martingale transport problems for cost {functionals $C \colon \R \times \mathcal P_1(\R) \to \R$, i.e.\ where the cost can also depend on the kernel} as in \eqref{WOT}, cf.\ \cite{BaBeHuKa17,GuMeNu17,CoVi19}. We also refer to \cite{GoJu18,FaGoPr19} for some unexpected applications of weak transport problems in analysis and probability.


The second main contribution of the article is the stability of optimal weak transport. Throughout the article we fix a compatible metric on $Y$, and for a real number $r\geq 1$ we denote by $\mathcal P_r(Y)$ the set of probability measures which integrate the function $d_Y(y_0,\cdot)^r$ for some $y_0 \in Y$, where $d_Y$ denotes a metric on $Y$. We endow $\mathcal P_r(Y)$ with the  the $r$-Wasserstein topology denoted by $\mathcal W_r$, and denote the space of continuous functions on $X$ by $\mathcal C(X)$.

%
%
%
%

\begin{theorem}[WOT Stability]\label{Thm stab WOT intro}
	Let $C,C_k:X\times \mathcal P_r(Y)\to[0,\infty)$, $k\in\N$, be continuous cost functions such that
	$C_k$ converges uniformly to $C$, and either one of the following holds
	\begin{enumerate}[label=(\alph*)]
		\item $\{C(x,\cdot)\colon x\in X\}\subset \mathcal C(P_r(Y)) $ is an equicontinuous family of convex functions,
		\item $\mu\in\mathcal P_r(X)$, and there is a constant $K>0$, $x_0\in X$, $y_0\in Y$ such that \[ \textstyle C(x,p) \leq K\left(1 + d_X(x,x_0)^r + \int_Y d_Y(y,y_0)^r\, p(dy)\right).\]
	\end{enumerate}
	Let {$\{\mu_k\}_{k \in \N}$, $\{\nu_k\}_{k \in \N}$ be sequences in $\mathcal P(X)$ and $\mathcal P_r(Y)$}, which converge weakly to $\mu$ and in $\mathcal W_r$ to $\nu$, respectively. For each $k \in \N$ let $\pi^k\in\Pi(\mu_k,\nu_k)$ be an optimizer of \eqref{WOT} with cost function $C_k$ {between the marginals $\mu_k$ and $\nu_k$}. If
	\[ \limsup_{k \to \infty} \int_X C_k(x,\pi^k_x)\, \mu_k(dx)< \infty,\]
	then any {weak} accumulation point of $\{\pi^k\}_{k\in\N}$ is an optimizer of \eqref{WOT} for the cost function $C$. In particular if the latter has a unique optimizer $\pi$, then $\pi^k\to \pi$ weakly. 
\end{theorem}

We now describe the main idea used in the proofs of Theorems \ref{Thm stab mot intro} and \ref{Thm stab WOT intro}.


\subsection{Monotonicity and the correct topology on the set of couplings }\label{sec:correct topology}

The article \cite{BaBePa18} investigates the optimal weak transport problem by essentially enlarging the original state space $X\times Y$ to $X\times \mathcal P(Y)$. We briefly review this idea since it points to the right notion of monotonicity which will be useful in proving the above stability results.

First we introduce the embedding map
\begin{align}\label{def J}
\begin{split}
J\colon &\mathcal P(X\times Y) \rightarrow \mathcal P(X\times\mathcal P(Y)),\\
			&\pi \mapsto {\proj_X(\pi)(dx) \, \delta_{\pi_x}(dp)},
\end{split}
\end{align}
{in words, if $(X,Y) \sim \pi$ then $(X,\pi_X) \sim J(\pi)$,
and its left-inverse, the $X\times Y$-intensity map $\hat I$,
\begin{align} \label{def hat I}
	\begin{split}
	\hat I\colon &\mathcal P(X \times \mathcal P(Y)) \to \mathcal P(X \times Y),\\
	& P \mapsto \int_{p \in \mathcal P(Y)}p(dy) \, P(dx,dp).
	\end{split}
\end{align}
}

Despite the fact that {$J$} is seldom continuous, it does enjoy a key property: it preserves the relative compactness of a set. As a consequence, one can easily obtain optimizers for the following suitable extension of \eqref{WOT} as soon as $C$ is lower semicontinuous:
\begin{align}\label{WOT'}\tag{WOT'}
	\inf_{P\in\Lambda(\mu,\nu)} \int_{X \times \mathcal P(Y)} C(x,p) \, P(dx,dp),
\end{align}
where $\Lambda(\mu,\nu)$ is the set of couplings $P \in \mathcal P(X\times\mathcal P(Y))$ 
{with $\hat I(P) \in \Pi(\mu,\nu)$.} 
When $C(x,\cdot)$ is furthermore convex, the extended problem \eqref{WOT'} is equivalent to the original one, and in addition,  \eqref{WOT} can be shown to admit an optimizer by means of the natural projection operator {$\hat I$} from $P(X\times\mathcal P(Y))$ onto $\mathcal P(X\times Y)$.

This idea of using an embedding (which preserves relative compactness) into a larger space can be appreciated in the following terms: On the original space $\mathcal P(X\times Y)$ we consider the initial topology of $J$ when the target space is given the weak topology. This initial topology has been studied in \cite{BaBeEdPi17,BaBaBeEd19} and given the name \emph{adapted weak topology}. One immediate observation is that the cost functional of  \eqref{WOT} is lower semicontinuous in this topology if $C$ is likewise lower semicontinuous.

Since under the stated conditions it makes no difference to work with \eqref{WOT} or \eqref{WOT'}, it helps to consider the latter simpler linear problem in order to inform our intuition. By analogy with optimal transport we  define 

\begin{definition}[$C$-monotonicity]\label{intro def:C-monotonicity}
	A coupling $\pi\in\Pi(\mu,\nu)$ is $C$-monotone iff there exists a $\mu$-full set $\Gamma\subset X$ such that for any finite number of points $x_1,\ldots,x_N\in\Gamma$ and $q_1,\ldots,q_N\in\mathcal P(Y)$ with $\sum_{i=1}^N \pi_{x_i} = \sum_{i=1}^N q_i$, we have
	$$\sum_{i=1}^N C(x_i,\pi_{x_i}) \leq \sum_{i=1}^N C(x_i,q_i).$$
\end{definition}
 
 It was shown under mild assumptions in \cite{BaBePa18} that optimality of $\pi$ for \eqref{WOT} implies $C$-monotonicity in the sense of Definition~\ref{intro def:C-monotonicity} above. The reverse implication(sufficiency) was shown to be true under the additional assumption that the cost function {$C$} is uniformly $\mathcal W_1$-Lipschitz {in the second argument}. In the present article we will generalize this result (and largely simplify the arguments) in Theorem \ref{thm:WOT optimality}. Once we are equipped with this necessary and sufficient criterion for optimality, the stability result Theorem \ref{Thm stab WOT intro} becomes a consequence of the fact that the notion of $C$-monotonicity is itself stable.
 
Although martingale optimal transport is a particular case of optimal weak transport, in this work we treat the two problems separately. 
{The reason for this lies in the fact that our approach to optimal weak transport requires regularity of the cost, which the `embedded cost' $C_M$, see \eqref{MOT subset WOT}, does not provide any longer.}
{To deal with the singular cost $C_M$, we refine the notion of $C$-monotonicity when it comes to martingale couplings and additionally employ techniques  which currently only work in dimension one.}
We define
\begin{definition}[Martingale $C$-monotonicity]\label{intro def:mart C-monotonicity}
	A coupling $\Pi_M(\mu,\nu)$ is martingale $C$-mono\-tone iff there exists a $\mu$-full set $\Gamma\subset \R^d$ such that for any finite number of points $x_1,\ldots,x_N\in\Gamma$ and $q_1,\ldots,q_N\in\mathcal P_1(\R^d)$ with $\sum_{i=1}^N \pi_{x_i} = \sum_{i=1}^N q_i$ and $\int_{\R^d} y \, q_i(dy) = x_i$, we have
	\begin{align*}
		\sum_{i=1}^N C(x_i,\pi_{x_i}) \leq \sum_{i=1}^N C(x_i,q_i).
	\end{align*}
\end{definition}
Then the key to proving Theorem \ref{Thm stab mot intro} boils down to two arguments: that martingale $C$-monotonicity is sufficient for optimality, and that this notion of monitonicity is itself stable. 

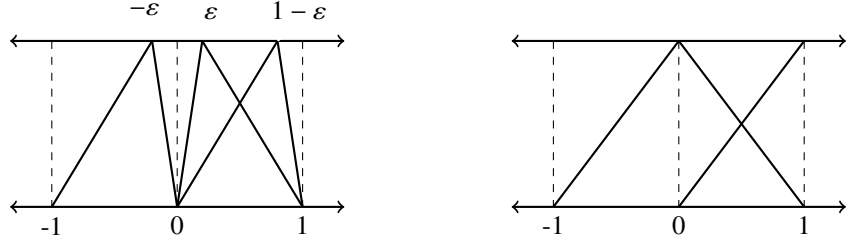
\begin{figure}
	\begin{tikzpicture}[scale=1.1]
	\draw[thick,<->] (-2,0)--(2,0) ;
	\draw[thick,<->] (-2,2)--(2,2);
	\draw[thick,-] (-0.3,2)--(-1.5,0) node[anchor=north ] {-1};
	\draw[thick,-] (-0.3,2)--(0,0);
	\draw[thick,-] (0.3,2)--(1.5,0);
	\draw[thick,-] (0.3,2)--(0,0);
	\draw[thick,-] (1.2,2)--(1.5,0);
	\draw[thick,white,-] (1.2,2)--(1.9,2.6) node[anchor=north east , black] {$1-\epsilon$};
	\draw[thick,white,-] (0.3,2.1)--(0.6,2.52) node[anchor=north east , black] {$\epsilon$};
	\draw[thick,white,-] (-0.3,2.1)--(-0.1,2.6) node[anchor=north east , black] {$-\epsilon$};
	\draw[thick,-] (1.2,2)--(0,0);
	\draw[dashed,-] (0,2)--(0,0) node[anchor=north ] {0};
	\draw[dashed,-] (1.5,2)--(1.5,0) node[anchor=north ] {1};
	\draw[dashed,-] (-1.5,2)--(-1.5,0);
	\draw[thick,<->] (4,0)--(8,0) ;
	\draw[thick,<->] (4,2)--(8,2);
	\draw[dashed,-] (6,2)--(6,0) node[anchor=north ] {0};
	\draw[dashed,-] (7.5,2)--(7.5,0) node[anchor=north ] {1};
	\draw[dashed,-] (4.5,2)--(4.5,0) node[anchor=north ] {-1};
	\draw[thick,-] (6,2)--(4.5,0) ;
	\draw[thick,-] (6,2)--(7.5,0) ;
	\draw[thick,-] (7.5,2)--(6,0) ;
	\end{tikzpicture}
	\caption{Left: Six possibe paths of the left-curtain coupling between a uniform distribution on the interval $[-1,1]$ and the measure $1/4(\delta_{-1}+\delta_{1}+ 2\delta_{0}) $, c.f.\ \cite[Example 2.11]{Ju16}.
	Right: Three paths contained in the support of this coupling (as can be seen from the left picture after sending $\epsilon\to 0$). These violate the ``left-curtain'' condition.}
	\label{figure}
\end{figure}

{The idea that a monotonicity viewpoint for martingale optimal transport has to be based on the disintegrations of a coupling, or put otherwise, that optimality cannot be read-off from the support of a martingale coupling, goes to Juillet \cite{Ju16}. This author in fact showed that for the left-curtain coupling (c.f.\ \cite{BeJu16}) the support of an optimizer may contain sub-optimal paths; see Figure \ref{figure} for an illustration. This is in stark contrast with classical optimal transport, where for say continuous bounded cost functions $c$ it is known that $c$-cyclical monotonicity of the support is equivalent to optimality. This failure ultimately underlies the need for our Definition \ref{intro def:mart C-monotonicity}. }

\subsection{Outline} The article is written so that the sections concerning optimal weak transport and martingale optimal transport are largely independent. Section \ref{sec:wot} deals with the stability of optimal weak transport, whereas Section \ref{sec:mot} explores the stability of martingale optimal transport. Section \ref{sec:wotot} studies the relationship between classical optimal transport and optimal weak transport. Finally Section  \ref{sec:epi} wraps up the article with a number of remarks, {followed by an appendix containing auxiliary results}.

\section{On the weak transport problem}\label{sec:wot}

Let us complement Definition \ref{intro def:C-monotonicity} with a more complete list of monotonicity properties:
{see Remark \ref{rem:cycl monotonicity} \ref{rem:cycl monotonicity 1} for a comparision of these definitions.}

\begin{definition}\label{def:C-monotonicity}~
\begin{enumerate}
	\item\label{it:CM2}  We call $\Gamma\subset X\times \mathcal P(Y)$ $C$-monotone iff for any finite number of points $$(x_1,p_1),\ldots,(x_N,p_N) \in \Gamma\text{  and }q_1,\ldots,q_N\in\mathcal P(Y)\text{ with }\sum_{i=1}^N p_i = \sum_{i=1}^N q_i,$$ we have
	$$\sum_{i=1}^N C(x_i,p_i) \leq \sum_{i=1}^N C(x_i,q_i).$$
	\item\label{it:CM3} A probability measure $P \in \mathcal P(X\times \mathcal P(Y))$, which is concentrated on a $C$-monotone set, is called $C$-monotone.
\end{enumerate}	
\end{definition}
{ The next theorem greatly extends the sufficiency of $C$-monotonicity first obtained in \cite[Theorem 5.5]{BaBePa18}.} We recall that $\mathcal P_r(Y)$ denotes the space of probability measures on $Y$ with finite $r$-th moment, { $r\geq 1$,} i.e., $p \in \mathcal P_r(Y)$ iff $p\in \mathcal P(Y)$ and
\[
	\int_Y d_Y(y,y_0)^r \, p(dy) < \infty
\]
for some $y_0 \in Y$, and thus all $y_0 \in Y$. It is well-known that $\mathcal P_r(Y)$ equipped with the topology of the Wasserstein metric $\mathcal W_r$ becomes a Polish space, where $\mathcal W_r(p,q)$ for $p,q\in\mathcal P_r(Y)$ is defined via
\[
	\mathcal W_r(p,q)^{r} = \inf_{\pi \in \Pi(p,q)} \int_{Y\times Y} d_Y(y_1,y_2)^r \, \pi(dy_1,dy_2).
\]
{Furthermore, we fix on $X \times \mathcal P_r(Y)$ the metric $d((x,p),(x',p'))^r := d_X(x,x')^r + \mathcal W_r(p,p')^r$ for $(x,p),(x',p') \in X \times \mathcal P_r(Y)$, and denote by $\mathcal P_r(X \times \mathcal P_r(Y))$ the space of probability measures on $X \times \mathcal P_r(Y)$, which finitely integrate $(x,p) \mapsto d((x,p),(x',p'))^r$ for given $(x',p') \in X \times \mathcal P_r(Y)$, equiped with the $\mathcal W_r$-metric associated with the metric $d$.
Similarly, we fix on $\mathcal P_r(Y)$ the Wasserstein metric $\mathcal W_r$, and denote by $\mathcal P_r(\mathcal P_r(Y))$, the space of probability measures on $ \mathcal P_r(Y)$, which finitely integrate $p \mapsto \mathcal W_r(p,p')^r$ for given $p' \in \mathcal P_r(Y)$, equiped with the $r$-th Wasserstein topology.
Note that these definitions are independent of the choice of $(x',p') \in X \times \mathcal P_r(Y)$ and $p' \in \mathcal P_r(Y)$, respectively.
}

Recall the definition of the set $\Lambda(\mu,\nu)$ given after \eqref{WOT'}, i.e.,
\begin{equation}\label{eq:def Lambda}
	\Lambda(\mu,\nu) := \left\{P \in \mathcal P(X\times \mathcal P(Y))\colon \hat I(P) \in \Pi(\mu,\nu)\right\}.
\end{equation}


\begin{theorem}\label{thm:WOT optimality}
	Let $\mu\in\mathcal P(X)$, $\nu\in\mathcal P_r(Y)$, $C\colon X\times \mathcal P_r(Y)\rightarrow \R$ measurable. Assume either of the following conditions:
	\begin{itemize}
	\item[(a)] {$\left\{ C(x,\cdot) \colon x \in X \right\} \subset \mathcal C(\mathcal P_r(Y))$ is equicontinuous continuous,} i.e.\
	{
		\[	\theta(\delta) := \sup\left\{ |C(x,p) - C(x,q)| \colon x \in X, (p,q) \in \mathcal P_r(Y)^2 \text{ s.t.} \mathcal W_r(p,q)^r \leq \delta \right\},	\]
	vanishes for $\delta \searrow 0$.
	}

\item[(b)] $\mu \in \mathcal P_r(X)$, $C$ is jointly continuous and there is $K\in\mathbb R,$ $x_0\in X$, $y_0\in Y$ s.t.\ for all $x\in X$: 
\[\textstyle C(x,p)\leq K\left(1+d_X(x,x_0)^r + \int d_Y(y,y_0)^r \, p(dy)\right ).\]
\end{itemize}	 
	Then $P \in \Lambda(\mu,\nu)$ is optimal for \eqref{WOT'} iff $P$ is $C$-monotone. Similarly, if $p\mapsto C(x,p)$ is convex, then $\pi\in\Pi(\mu,\nu)$ is optimal for \eqref{WOT} iff $\pi$ is $C$-monotone.
\end{theorem}
{
\begin{remark}
	Note that under the assumptions of Theorem \ref{thm:WOT optimality} $C$-monotonicity of $\pi \in \Pi(\mu,\nu)$ yields $C$-monotonicity of $J(\pi) \in \Lambda(\mu,\nu)$, thus, $J(\pi)$ is optimal for \eqref{WOT'} and particularly $\pi$ is optimal for \eqref{WOT}. {We will complete this discussion in Remark \ref{rem:cycl monotonicity}. }

	At the same time, we see in Section \ref{sec:epi} Example \ref{ex:1} that convexity of $p \mapsto C(x,p)$ is necessary for optimal couplings being $C$-monotone. Example \ref{ex:2} shows that additional regularity (to lower semicontinuity and boundedness) of the cost $C$ is required for $C$-monotonicity being a sufficient optimality criterion.
\end{remark}}

\begin{proof}
	Let $P$ be $C$-monotone concentrated on the $C$-monotone set $\Gamma$. Fix $P'\in \Lambda(\mu,\nu)$. We argue as in \cite{Be12} for classical (linear) optimal transport. Take any iid sequences $(X_n)_{n\in\N}$ of $X$-valued random variables, and any iid sequences  $(Y_n)_{n\in\N}$, $(Z_n)_{n\in\N}$ of $\mathcal P(Y)$-valued random variables, on some probability space $(\Omega,\mathbb{P})$, with
	 \[(X_n,Y_n)\sim P,\quad (X_n,Z_n)\sim P'.\]
	 In particular by the law of large numbers, we find $\mathbb P$-almost surely
	 \begin{multline}
		\label{eq:lln}
		 \int_{X \times \mathcal P_r(Y)} C(x,p) P'(dx,dp) - \int_{X \times \mathcal P_r(Y)} C(x,p) P(dx,dp)
		 \\ =  \lim_{N \to \infty} \frac{1}{N} \sum_{n=1}^N C(X_n,Z_n) - C(X_n,Y_n).
	 \end{multline}
	 Note that for any function $g \in \mathcal C(Y)$, which is majorized by $y\mapsto 1 + d_Y(y,y_0)^r$, we have
	 \begin{align*}
	 	\mathbb{E}[Y_n(g)] = \nu(g),
	 \end{align*}
	 {where $p(g)$ for $p \in \mathcal P(Y)$ denotes the integral $\int_Y g(y) \, p(dy)$.}
	 Then the law of large numbers implies almost surely
	 \[\lim_{N \to \infty} \frac{1}{N}\sum_{n=1}^N Y_n(g) = \nu(g) = \lim_{N \to \infty} \frac{1}{N}\sum_{n=1}^N Z_n(g).\]
	 By standard separability arguments, we find $\mathbb{P}$-almost surely
	 \begin{equation}
		 \label{eq:marginals converge}
		 \lim_{N\to \infty} \frac{1}{N}\sum_{n=1}^N Y_n = \nu = \lim_{N \to \infty} \frac{1}{N}\sum_{n=1}^N Z_n\quad\mathbb{P}\text{-a.s.}
	 \end{equation}
	 where convergence holds in $\mathcal W_r$. Let $\omega \in \Omega$ be in a $\mathbb P$-full set s.t.
	 \[\lim_{N\to \infty} \mathcal W_r\left(\frac{1}{N}\sum_{n=1}^N Y_n(\omega),\frac{1}{N}\sum_{n=1}^N Z_n(\omega)\right) = 0,\]
	 and $(X_n,Y_n)(\omega) \in \Gamma$ for all $n\in\N$.
From now on we omit the $\omega$ argument. For each $N\in\N$, we denote a $\mathcal W_r$-optimal coupling in $\Pi(\frac{1}{N}\sum_{n=1}^N Z_n,\frac{1}{N}\sum_{n=1}^N Y_n)$ by $\chi^N$, { i.e.\
\[
	 \mathcal W_r\left(\frac{1}{N} \sum_{n = 1}^N Z_n, \frac{1}{N} \sum_{n = 1}^N Y_n \right)^r = \int_{Y \times Y} d_Y(z,y)^r \, \chi^N(dz,dy).
\]
}
We denote by $\{\chi^N_z \}_{z\in Y}$ a regular disintegration of $\chi^N$ given its projection in the first coordinate (marginal). Defining $\chi^N Z_n (dy):=  \int_{z\in Y} Z_n(dz)\,\chi^N_z(dy)$, we find
	\begin{align}\label{eq:WOT opt1}
	\frac{1}{N} \sum_{n=1}^N \mathcal W_r(Z_n, \chi^N Z_n)^r \leq \frac{1}{N} \sum_{n=1}^N \int_{Y^2} d_Y(z,y)^r \chi^N_z(dy) Z_n(dz) = \mathcal W_r\left({\frac{1}{N}\sum_{n=1}^N Z_n, \frac{1}{N}\sum_{n=1}^N Y_n}\right)^r.
\end{align}
	 Moreover, $\sum_{n=1}^N \chi^N Z_n = \sum_{n=1}^N Y_n$, so by $C$-monotonicity
	 \begin{align}
	\label{eq:WOT opt2}
	 \frac{1}{N} \sum_{n=1}^N C(X_n,\chi^N Z_n) - C(X_n,Y_n) \geq 0.
	 \end{align}
	 In the case of $(a)$, {we apply Lemma \ref{lem:mod_cont_wasser} and get $\tilde \theta\colon \R^+ \rightarrow \R^+$, which is continuous, concave, and vanishing at $0$, }{and which we rename $\theta$ for simplicity.}
	 Hence, by Jensen's inequality,
	 \[
	 	\frac{1}{N}\sum_{n=1}^N \theta\left(\mathcal W_r(Z_n,\chi^N Z_n)^r\right) \leq 
	 	\theta\left(\frac{1}{N}\sum_{n=1}^N \mathcal W_r(Z_n,\chi^N Z_n)^r \right) \leq \theta\left( \mathcal W_r\left(\frac{1}{N}\sum_{n = 1}^N Z_n, \frac{1}{n} \sum_{n = 1}^N Y_n\right)^r\right),
	 \]
	 which vanishes as $N\to +\infty$. Using $C$-monotonicity of $P$ and uniform continuity, we obtain $\mathbb P$-almost surely
	 \begin{multline}
		\label{eq:WOT  opt3}
			\frac{1}{N} \sum_{n=1}^N C(X_n,Z_n) - C(X_n,Y_n) 
		\\	= \frac{1}{N} \sum_{n=1}^N C(X_n,Z_n) - C(X_n,\chi^N Z_n) + \frac{1}{N} \sum_{n=1}^N C(X_n,\chi^N Z_n) - C(X_n,Y_n)
		\\	\geq -\frac{1}{N} \sum_{n=1}^N \theta(\mathcal W_r(Z_n,\chi^NZ_n))
		\\	\geq -\theta\left(\Bigg(\mathcal W_r\Bigg(\frac{1}{N}\sum_{n=1}^N Y_n(\omega),\frac{1}{N}\sum_{n=1}^N Z_n(\omega)\Bigg)\right)\rightarrow 0.
	 \end{multline}
	{In the case of $(b)$, {we define the random variable $P_N'$} taking values in $X\times \mathcal P(Y)$ by
	\[ P_N' := \frac{1}{N} \sum_{n=1}^N \delta_{X_n, \chi^N Z_n}. \]
	By \eqref{eq:WOT opt1} we find $\mathbb{P}$-almost surely 
	\[
		\mathcal W_r(P_N', P')^r \leq \frac{1}{N}\sum_{n=1}^N \mathcal W_r(Z_n,\chi^N Z_n) \leq \mathcal W_r\left(\frac{1}{N} \sum_{i=1}^N Z_n, \frac{1}{N} \sum_{n=1}^N Y_n\right)^r,
	\]
	the right-hand side of which converges a.s.\ to zero as we have already seen {in \eqref{eq:marginals converge}}.\footnote{Here, the Wasserstein distance on the left-hand side of the equation is taken on the space $X \times \mathcal P_r(Y)$ w.r.t. the metric $d((x_1,p_1),(x_2,p_2))^r = d_X(x_1,x_2)^r + \mathcal W_r(p_1,p_2)^r$.} Then, continuity and growth of $C$ yields
	\begin{align}
		\label{eq:WOT opt4}
		\int_{X \times \mathcal P_r(Y)} C(x,p) \, P'(dx,dp) - \int_{X \times \mathcal P_r(Y)} C(x,p) \, P_N'(dx,dp) \rightarrow 0.
	\end{align}	
Hence, in both cases we have $\mathbb{P}$-almost surely}
	\begin{multline*}
		\int_{X \times \mathcal P_r(Y)} C(x,p) \, P'(dx,dp) - \int_{X \times \mathcal P_r(Y)} C(x,p) \, P(dx,dp) \\ \geq \liminf_{N \to \infty} \frac{1}{N} \sum_{n=1}^N C(X_n,Z_n) - C(X_n,Y_n) \geq 0,
	\end{multline*}
	{where we used \eqref{eq:lln},\eqref{eq:WOT opt2}, \eqref{eq:WOT  opt3}, and \eqref{eq:WOT opt4}.}
\end{proof}

\subsection{Stability of \texorpdfstring{$C$}{}-Monotonicity}

Recall the embedding of \eqref{def J}
\begin{align*}
	J\colon &\mathcal P(X\times Y) \rightarrow \mathcal P(X\times\mathcal P(Y)),\\
			&\pi \mapsto {\proj_X(\pi)(dx) \, \delta_{\pi_x}(dp).}
\end{align*}
The intensity $I(Q)$ of some measure $Q\in\mathcal P(\mathcal P(Y))$ is uniquely defined as the probability measure $I(Q) \in \mathcal P(Y)$ with
\[ I(Q)(f) = \int_{\mathcal P(Y)} \int_Y f(y) p(dy) \, Q(dp)\quad \forall f \in \mathcal C_b(Y), \]
{that is $I(Q)(dy) = \int_{p \in \mathcal P(Y)} p(dy) \, Q(dp)$.}

\begin{remark}\label{rem:cycl monotonicity} In the light of this embedding it appears to be natural to consider $C$-mono\-tonicity on the enhanced space $X\times \mathcal P(Y)$.
\begin{enumerate}[label = (\alph*)]
\item \label{rem:cycl monotonicity 1} {$\pi\in \Pi(\mu,\nu)$ is $C$-monotone iff $J(\pi)$ is $C$-monotone:
On the one hand, if $\pi$ is $C$-monotone it is possible to find a measurable set $\tilde \Gamma \subseteq X$ such that $\mu(\tilde \Gamma) = 1$ which defines via
\[\Gamma = \left\{ (x,p) \in \tilde \Gamma \times \mathcal P(Y) \colon p = \pi_x  \right\},\]
a $C$-monotone set. Therefore, equivalently to Definition \ref{def:C-monotonicity}, we can demand that there exists a $C$-monotone set $\Gamma\subset X\times \mathcal P(Y)$ such that $(x,\pi_x)\in \Gamma$ for $\mu$-almost every $x\in X$.

On the other hand, if $J(\pi)$ is $C$-montone, then there exists a $C$-montone set $\Gamma$ with
\[
	1 = J(\pi)(\Gamma) = \mu\left( \left\{ x \in X \colon (x,\pi_x) \in \Gamma \right\}\right).
\]
Consider the $\mu$-full, analytically measurable set $\{ x \in X \colon (x,\pi_x) \in \Gamma\}$.
As analyitcally measurable sets are universally measurable, it admits a Borel measurable subset $\tilde \Gamma$
with $\mu(\tilde \Gamma) = 1$.
Thus, $\pi$ is $C$-monotone (on $\tilde \Gamma$) in the sense of Definition \ref{intro def:C-monotonicity}.}
\item \label{rem:cycl monotonicity 2} If $\Gamma \subset X \times \mathcal P(Y)$ is $C$-monotone, and $C\colon X\times \mathcal P(Y) \rightarrow \R$ is convex in the second argument, i.e., {for all $x \in X,~ (p,q) \in \mathcal P(Y)^2,$ and $\alpha \in [0,1]$}
{$$C(x,\alpha p + (1- \alpha) q) \leq \alpha C(x,p) + (1 - \alpha) C(x,q),$$}
then the enlarged set
$$\tilde \Gamma :=  \left\{ \Bigg(x, \frac{1}{k} \sum_{i=1}^k p_i\Bigg)  \colon x \in X,~ (x,p_i) \in \Gamma, ~ i=1,\ldots, k \in \N\right\}$$
is also $C$-monotone. { Likewise,} if $C(x,\cdot)$ is further continuous for all $x\in X$, then
$$\hat \Gamma := \left\{(x,p) \in X \times \mathcal P(Y)\colon x \in X, p \in \overline{\text{co}(\Gamma_x)} \right\},$$
is $C$-monotone, where  $\overline{\text{co}}$ stands for the closed convex hull  {and $\Gamma_x$ denotes the $x$-fibre of $\Gamma$, that is $\{p \in \mathcal P(Y)\colon (x,p) \in \Gamma \}$.}
\item \label{rem:Lambda Claus constraints}
{ We observe that the set $\Lambda(\mu,\nu)$ (see \eqref{eq:def Lambda})} can be characterized by a family of continuous functions $\mathcal F\subset \mathcal C(X\times \mathcal P_r(Y))$: $P \in \Lambda(\mu,\nu)$ if and only if
\begin{align*}
	\int_{X\times\mathcal P(Y)} f(x) \, P(dx,dp) =  \int_X f(x) \, \mu(dx)&\quad \forall f \in \mathcal C_b(X),\\
	\int_{X\times\mathcal P(Y)} \int_Y g(y) \, p(dy) \, P(dx,dp) = \int_Y g(y) \, \nu(dy)&\quad \forall g \in \mathcal C_b(Y).
\end{align*}
As a further observation we have the equivalence of $C$-monotonicity as in Definition~\ref{def:C-monotonicity} and $C$-finite optimality under the linear constraints $\mathcal F$ {
\begin{align*}
	\mathcal F = \Big\{f \in \mathcal C_b(X\times \mathcal P(Y)) \colon &\exists  g \in \mathcal C_b(X),~h\in \mathcal C_b(Y)\\
	&\textstyle\text{ s.t. }f(x,p) \equiv g(x) \text{ or } f(x,p) = \int_Y h(y)\,p(dy) \Big\},
\end{align*} 
which was introduced in \cite[Definition 1.2]{BeGr14}.}
\end{enumerate}	
\end{remark}

\begin{theorem}\label{thm:WOT monotonicity}
	Let $C\colon X \times \mathcal P_r(Y)\to {[0,\infty]}$ be measurable and $P^*\in\mathcal P_r(X\times \mathcal P_r(Y))$ optimal for \eqref{WOT'} with finite value. Then $P^*$ is $C$-monotone. In particular, if $C$ satisfies for all $x\in X$ and $Q \in \mathcal P(\mathcal P(Y))$
	\begin{align}\label{eq:integralconvexity}
		C(x,I(Q)) \leq \int_{\mathcal P(Y)} C(x,p) \, Q(dp),
	\end{align}
	then any optimizer $\pi^*$ of \eqref{WOT} with finite value is $C$-monotone. 
\end{theorem}

\begin{proof}
	The first assertion is a consequence of Remark~\ref{rem:cycl monotonicity}.\ref{rem:Lambda Claus constraints} and \cite[Theorem 1.4]{BeGr14}.
	To show the second assertion, let $P\in \Lambda(\mu,\nu)$. Then $I(P_x) \, \mu(dx) \in \Pi(\mu,\nu)$ and by \eqref{eq:integralconvexity}
	\begin{align*}
		\int_{X\times \mathcal P(Y)} C(x,p)\, P(dx,dp) \geq \int_{X} C(x,I(P_x)) \, \mu(dx).
	\end{align*}
	Hence, $J(\pi^*)$ is optimal for \eqref{WOT'}. 
	{We deduce from the first part of the proof combined with Remark \ref{rem:cycl monotonicity} \ref{rem:cycl monotonicity 1} $C$-monotonicity of $\pi$.}
\end{proof}

The assumption that $C$ is lower bounded is as a matter of fact not necessary to deduce $C$-monotonicity in the classical optimal weak transport setting, cf. \cite[Theorem~5.2]{BaBePa18}. Note that \eqref{eq:integralconvexity} holds when $C(x,\cdot)$ is lower semicontinuous and convex. { Indeed, by similar approximation arguments as in Theorem~\ref{thm:WOT optimality} we find for any $Q \in \mathcal P(\mathcal P(Y))$ a sequence of measures $\{p_k\}_{k\in\N}\subset \mathcal P(Y)$ such that
\[\frac{1}{N}\sum_{i=1}^N p_i \rightarrow I(Q)\quad \text{in }\mathcal W_r, \quad \frac{1}{N}\sum_{i=1}^N C(x,p_i) \rightarrow \int_{X \times \mathcal P(Y)} C(x,p) \, Q(dp).\]
Thus,
\begin{align*}
	\int_{X \times \mathcal P(Y)} C(x,p) Q(dp) &= \lim_{N \to \infty} \frac{1}{N} \sum_{i=1}^N C(x,p_i) \\ &\geq \liminf_{N \to \infty} C\left(x,\frac{1}{N}\sum_{i=1}^Np_i\right)\geq C\left(x,I(Q)\right).
\end{align*}
}

\begin{lemma}\label{lem:approximative sequence}
	Let $p_i, m_i \in \mathcal P_r(Y)$, $i = 1,\ldots, N$, with $\sum_{i=1}^N p_i = \sum_{i=1}^N m_i$, and $\{p_1^k,\ldots,p_N^k\}$, $k\in\N$, be sequences {in} $\mathcal P_r(Y)$ with $p_i^k \to p_i \text{ in } \mathcal W_r$.
	Then there exist approximative sequences $\{m_1^k,\ldots,m_N^k\}_{k \in \N}$ of competitors, i.e.,
	$$ \sum_{i=1}^N p_i^k = \sum_{i=1}^N m_i^k \text{ and } m_i^k \rightarrow m_i \text{ in }\mathcal W_r.$$
\end{lemma}

\begin{proof}
	Since $\sum_{i=1}^N m_i = \sum_{i=1}^N p_i$, we find sub-probability measures $m_{i,j}${, with the property that $m_{i,j}(A)\leq\min\{p_i(A),m_j(A)\}$ for all $A$ measurable, and}
	\begin{align*}
		m_j = \sum_{i=1}^N m_{i,j},\quad p_i = \sum_{j=1}^N m_{i,j}.
	\end{align*}	
	Denote by $(\chi^{k,i}_z)_{z\in Y}$ a regular disintegration of a $\mathcal W_r$-optimal transport plan $\chi^{k,i} \in \Pi(p_i,p_i^k)$ w.r.t. to its first marginal $p_i$. Let $i,j \in \{1,\ldots, N\}$ and define
	\begin{align*}
		m_{i,j}^k(dy) := \int_{ z \in Y} \chi^{k,i}_z(dy) \, m_{i,j}(dz),\quad m_i^k := \sum_{j=1}^N m_{j,i}^k.
	\end{align*}
	Since
	\begin{multline*}
		\sum_{j = 1}^N \mathcal W_r(m_{i,j}^k,m_{i,j})^r\leq \sum_{j = 1}^N \int_{Y \times Y}  d_Y(z,y)^r \, \chi^{k,i}_z(dy) \, m_{i,j}(dz)
	\\	\leq \int_{Y \times Y}  d_Y(z,y)^r\, \chi^{k,i}_z(dy) \, p_i(dz)=\mathcal W_r(p_i^k,p_i)^r,
	\end{multline*}
we deduce the convergence of $m_{i,j}^k$ to $m_{i,j}$, and in consequence, the convergence of $m_i^k$ to $m_i$ in $\mathcal W_r$. Finally observe that 
\[
	\sum_{i = 1}^N m_i^k=\sum_{i = 1}^N \sum_{j = 1}^N m_{j,i}^k= { \sum_{j=1}^N \int_{z \in Y} \chi^{k,j}_z(dy) \, p_j(dz)=\sum_{j=1}^N p_j^k}
\]
so indeed {$\{m_1^k,\ldots,m_N^k\}$ are feasible competitors of $\{p_1^k,\ldots,p_N^k\}$ such
that for $i = 1,\ldots,N$, $m_i^k \to m_i$ in $\mathcal W_r$.}
\end{proof}

\begin{lemma}\label{lem:closed Gamma}
	Let $C \in \mathcal C(X\times \mathcal P_r(Y))$, $\epsilon \geq 0$, and $N\in\N$. Then the set
	\begin{align}\label{def:GammaNepsilon}
	\begin{split}
		\Gamma_N^\epsilon := \Bigg\{(x_i,p_i)_{i=1}^N \in (X\times\mathcal P_r(Y))^N \Big|& \forall m_1,\ldots, m_N \in\mathcal P_r(Y) \text{ s.t. } \sum_{i=1}^N p_i = \sum_{i=1}^N m_i,\\
		&\text{ we have } \sum_{i=1}^N C(x_i,p_i) \leq \sum_{i=1}^N C(x_i,m_i) + \epsilon \Bigg\}
	\end{split}
	\end{align}
	is a closed subset of $(X\times \mathcal P_r(Y))^N$.
\end{lemma}

\begin{proof}
	Take any convergent sequence $(x_i^k,p_i^k)_{i=1}^N \in \Gamma_N^\epsilon$, $k\in\N$, such that
	$$ x_i^k \rightarrow x_i \text{ in } X,\quad p_i^k \rightarrow p_i \text{ in }\mathcal W_r.$$
	Assume that $(x_i,m_i)_{i=1}^N$ is a competitor, i.e., $\sum_{i=1}^N p_i = \sum_{i=1}^N m_i$.
	Lemma~\ref{lem:approximative sequence} provides an approximative sequence of competitors, and by continuity of $C$ we conclude.
\end{proof}

The key ingredient towards stability of \eqref{WOT} is the following result concerning stability of the notion of $C$-monotonicity.


\begin{theorem}\label{thm:stability WOT}
	Let $C,C_k\in \mathcal C(X\times \mathcal P_r(Y))$, $k\in\N$, and $C_k$ converges uniformly to $C$. If $P,P^k\in \mathcal P_r(X\times\mathcal P_r(Y))$, $k\in\N$, such that
	\begin{enumerate}[label=(\alph*)]
		\item for all $k\in\N$ the measure $P^k$ is $C_k$-monotone,
		\item the sequence $(P^k)_{k\in\N}$ converges to $P$,
	\end{enumerate}
then $P$ is $C$-monotone. Moreover, if $\pi,\pi^k\in {\mathcal P_r}(X\times Y)$ and $C_k$ is convex in the second argument, $k\in\N$, such that
	\begin{enumerate}[label=(\alph*')]
		\item for all $k\in\N$ the measure $\pi^k$ is $C_k$-monotone,
		\item the sequence $(\pi^k)_{k\in\N}$ converges to $\pi$,
	\end{enumerate}
	then $\pi$ is $C$-monotone.
\end{theorem}

\begin{proof}
	The aim is to construct a $C$-monotone set $\Gamma$ on which $P$ is concentrated. So, we write $P^{k,\otimes N}$ and $P^{\otimes N}$ for the $N$-fold product measure of $P^k$ and $P$ where $N\in\N$. By $C_k$-monotonicity and uniform convergence we find for any $\epsilon > 0$ a natural number $k_0$ such that $P^{k,\otimes N}$, $k\geq k_0$, is concentrated on $\Gamma^\epsilon_N$, see \eqref{def:GammaNepsilon}. Lemma~\ref{lem:closed Gamma} combined with the Portmanteau theorem yield that $P^{\otimes N}$ is concentrated on $\Gamma_N^\epsilon$:
	\begin{align*}
	1 = \limsup_k P^{k,\otimes N}(\Gamma^\epsilon_N) \leq P^{\otimes N}(\Gamma^\epsilon_N) = 1.
	\end{align*}
	As a consequence, we find that $P^{\otimes N}$ gives full measure to the closed set $\Gamma_N := \Gamma^0_N$. 
	{Sets of the form $\bigotimes_{i=1}^N O_i$ where $O_i$ is open in $X\times \mathcal P_r(Y)$
	form a basis of the product topology on $(X \times \mathcal P_r(Y))^N$.}
	Hence, we can cover the open set $\Gamma_N^c$ by countably many sets of the form $\bigotimes_{i=1}^N O_i$ where $O_i$ is open in $X\times \mathcal P_r(Y)$,
	\begin{align*}
		\Gamma_N^c = \bigcup_{k\in \N} \bigotimes_{i=1}^N O_{i,k}.
	\end{align*}
	In particular, we deduce for any $k\in\N$
	\begin{align*}
		0 = P^{\otimes N}\left(\bigotimes_{i=1}^N O_{i,k}\right) = \prod_{i=1}^N P(O_{i,k}).
	\end{align*}
	We find open sets $A_N$ such that
	\begin{gather*}
		A_N := \bigcup_{\substack{k \in \N, i \in \{1,\ldots,N\} \\ P(O_{i,k}) = 0}} O_{i,k},\quad P(A_N) = 0,\\ 
		\Gamma_N^c \subset \bigcup_{i=1}^N \left(X \times \mathcal P_r(Y) \right)^{i-1} \times A_N \times \left( X \times \mathcal P_r(Y) \right)^{N-i}.
	\end{gather*}
	Since $N\in\N$ was arbitrary we define the closed and $C$-monotone set
	\begin{align*}
		\Gamma := \bigg(\bigcup_{N\in\N} A_N\bigg)^c,\quad P(\Gamma) = 1,\quad \Gamma^N \subset \Gamma_N.
	\end{align*}
	With the taken precautions it poses no challenge to verify that $P$ is $C$-monotone on $\Gamma$. 
	
	To show the second assertion, we embed $\pi^k\in\mathcal P(X\times Y)$ into $\mathcal P(X\times \mathcal P(Y))$ owing to the map $J$. {Note that $\Lambda(\mu,\nu)$ is a closed subset of $\mathcal P_r(X \times \mathcal P_r(Y))$. By the same line of argument as in \cite[Lemma 2.6]{BaBePa18} we find relative compactness, thus compactness of $\Lambda(\mu,\nu)$.}	
	{ Therefore,} we find an accumulation point $P\in \mathcal P(X\times \mathcal P(Y))$ of $(J(\pi^k))_{k\in\N}$. Note that
	\[
		\mu(dx) \, I(P_x) =: \pi \in \Pi(\mu,\nu),
	\]
	determines a coupling, which is likewise an accumulation point of $\{ \pi^k\}_{k\in\N}$. Since $P$ is concentrated on the $C$-monotone set $\Gamma$, we find for any $x\in X$ such that $P_x(\Gamma_x) = 1$ a sequence of measures $p_i^x \in \Gamma_x \subset \mathcal P(Y),~i\in\N$, with
	$$q_n^x := \frac{1}{n}\sum_{i=1}^n p_i^x \rightarrow I(P_x) = \pi_x,\quad n\rightarrow \infty,\quad  \text{ in }\mathcal W_r.$$
	By Remark~\ref{rem:cycl monotonicity} \ref{rem:cycl monotonicity 2} we know that $(x,q_n^x)$ is contained in the $C$-monotone set $\Gamma$. By closure of $\Gamma$ we conclude $(x,\pi_x) \in \Gamma$ for $\mu$-a.e. $x$, and $C$-monotonicity of $\pi$.
\end{proof}

From Theorems \ref{thm:WOT optimality} and \ref{thm:stability WOT} we easily deduce the following corollary, which has Theorem \ref{Thm stab WOT intro} in the introduction as a particular case:


\begin{corollary}
	Let $C,C_k\in \mathcal C(X\times \mathcal P_r(Y))$, $k\in\N$, be non-negative cost functions such that $C_k$ converges uniformly to $C$ and one of the following holds:
	\begin{enumerate}[label=(\alph*)]
		\item {$\{ C(x,\cdot) \colon x \in X \}$ is equicontinuous},
		\item $\mu \in \mathcal P_r(X)$ and there is a constant $K>0$, $x_0\in X$, $y_0\in Y$ such that for all $(x,p) \in X \times \mathcal P(X)$ $$\textstyle C(x,p) \leq K\left(1+d_X(x,x_0)^r + \int_Y d_Y(y,y_0)^r \, p(dy)\right).$$
	\end{enumerate}
	Let $\{ \mu_k \}_{k \in \N}$ and $\{ \nu_k\}_{k \in \N}$ be two sequences of probability measures on $\mathcal P(X)$ and $\mathcal P_r(Y)$, respectively, where $\mu_k$ converges weakly to $\mu$, and $\nu_k$ converges in $\mathcal W_r$ to $\nu$. Let $P^k\in\Lambda(\mu_k,\nu_k)$ be optimizers of \eqref{WOT'} with cost function $C_k$ {between the marginals $\mu_k$ and $\nu_k$}. If
	$$\limsup_k P^k(C_k) < \infty,$$
	then any {weak }accumulation point of $\{P^k\}_{k \in \N}$ is an optimizer of \eqref{WOT'} for the cost $C$.
	
	If moreover $C_k(x,\cdot)$ and $C(x,\cdot)$ are convex, then an analogous statement holds in the case of \eqref{WOT}.
\end{corollary}


\section{Stability of martingale optimal transport}\label{sec:mot}
In this section we consider the martingale optimal transport problem~\eqref{MOT}, and $X = Y = \R^d$. A generalization of $c$-cyclical monotonicity under additional linear constraints were suggested by \cite{BeGr14,Za14}, which also encompass \eqref{MOT}. {For \eqref{MOT}},
the set of linear constraints $\mathcal F_M \subset \mathcal C(\R^d\times \R^d)$ takes the shape
\begin{align*}
	\mathcal F_M := \left\{ f \in \mathcal C(\R^d \times \R^d)\colon f(x,y) = g(x)(y-x) \text{ and } g \in \mathcal C_b(\R^d)  \right\}.
\end{align*}

\begin{definition}\label{def:martingale monotonicity}~
\begin{enumerate}
	\item\label{it:MM1} A measures $\alpha' \in \mathcal P(\R^d\times\R^d)$ is called a $\mathcal F_M$-competitor of $\alpha \in \mathcal P(\R^d\times\R^d)$ iff their marginals coincide and $\alpha(f) = \alpha'(f)$ for all $f\in\mathcal F_M$.
	\item\label{it:MM2} We call $\Gamma\subset \R^d\times \R^d$ $(c,\mathcal F_M)$-monotone iff for any probability measure $\alpha$, finitely supported on $\Gamma$, and any competitor $\alpha'$, we have $\alpha(c) \leq \alpha'(c)$.
	\item\label{it:MM3} A martingale coupling $\pi \in \Pi_M$, which is supported on a $(c,\mathcal F_M)$-monotone set, is called $(c,\mathcal F_M)$-monotone.
\end{enumerate}	
\end{definition}

\begin{remark}\label{rem:martingale monotonicity}
	Definition~\ref{def:martingale monotonicity}, Point \eqref{it:MM2}, implies for any $(c,\mathcal F_M)$-monotone set $\Gamma$ that for all sequences $(x_1,p_1),\ldots,(x_N,p_N)\in \mathcal \R^d \times \mathcal P_1(\R^d)$ with $p_i$ finitely supported on the fibre $\Gamma_{x_i}$ that
	\begin{align}\label{eq:rem mm}
		\sum_{i=1}^N \int_{\R^d} c(x_i,y) \, p_i(dy) \leq \sum_{i=1}^N \int_{\R^d} c(x_i,y) \, q_i(dy),
	\end{align}
	where $q_1,\ldots,q_N\in \mathcal P_1(\R^d)$, $\sum_{i=1}^N q_i = \sum_{i=1}^N p_i$ and $\int_{\R^d} y \, p_i(dy) = \int_{\R^d} y \, q_i(dy)$. 
	
	On the other hand, suppose $\Gamma$ is a set such that \eqref{eq:rem mm} holds for all $N\in\N$, and collections $(x_1,p_1),\ldots,(x_N,p_N)\in \R^d \times \mathcal P_1(\R^d)$ with $p_i$ finitely supported on $\Gamma_{x_i}$.
	Given any measure $\alpha \in \mathcal P(\Gamma)$ supported on the finite set $\{(x_1,y_1),\ldots,(x_n,y_n)\}$, we have optimality of $\{(x_1,\alpha_{x_1}),\ldots,(x_n,\alpha_{x_n})\}$ where $\alpha_{x_i}(dy) = \frac{\alpha(x_i,dy)}{\alpha(x_i\times \R^d)}$ over all competing sequences as in \eqref{eq:rem mm}. Therefore, $\tilde \alpha(dx,dy) := \sum_{i=1}^n \delta_{x_i}(dx) \, \alpha_{x_i}(dy)$ defines an optimal coupling between its marginals under all competitors, i.e., for all $\gamma \in \Pi(\proj_1(\tilde \alpha),\proj_2(\tilde \alpha))$ with $\int_{\R^d} y \gamma_x(dy) = \int_{\R^d} y \, \alpha_x(dy)$ we have
	\begin{align} \label{eq:rem mm2}
		\frac{1}{n} \sum_{i=1}^n c(x_i,y_i) = \tilde \alpha(c) \leq \gamma(c).
	\end{align}
	By the duality theorem of linear programming, we find dual optimizers of the linear program given by \eqref{eq:rem mm2}, i.e., maximizers of
	\begin{align*}
		\max_{f(x)+g(y)+\Delta(x)\cdot(x-y) \leq c(x,y)} \sum_{i=1}^N f(x_i) + \sum_{i=1}^N \alpha_{x_i}(g) + \sum_{i=1}^N \Delta(x_i) \cdot \int_{\R^d} (x_i-y)\alpha_{x_i}(dy).
	\end{align*}
	The complementary  slackness condition, which reads here as
	$$(f(x_i) + g(y_j) + \Delta(x_i)\cdot (x_i - y_j)) \tilde \alpha(x_i,y_j) = 0\quad \forall 1\leq i,j\leq n,$$
	 yields that optimality of the dual optimizer is independent of the definitive choice of the measure $\alpha$ -- as long as $\supp \alpha \subset \supp \tilde \alpha$.  Hence, we deduce {$(c,\mathcal F_M)$-monotonicity} of $\alpha$, and $\Gamma$ is $(c,\mathcal F_M)$-monotone. In Definition~\ref{def:weak martingale monotonicity} we introduce a notion of martingale $C$-monotonicity for weak transport costs which by above reasoning naturally extends $(c,\mathcal F_M)$-monotonicity, see Definition \ref{def:martingale monotonicity}, to weak transport costs.
\end{remark}

We will see in Lemma \ref{lem:c,mart mon equiv C,mart mon} that under given conditions $(c,\mathcal F_M)$-monotonicity of a coupling is equivalent to martingale $C$-monotonicity (cf.\ Definition \ref{intro def:mart C-monotonicity}).

	By \cite{BeGr14}, optimizers of \eqref{MOT} are concentrated on $(c,\mathcal F_M)$-monotone sets. If $c$ is continuous, then the reverse implication was shown in one dimension by Beiglb\"ock and Juillet \cite{BeJu16} and Griessler \cite{Gr16}, but for arbitrary dimensions $d\in\N$ it remains unanswered. 

Let us regard two natural generalizations of \eqref{MOT}:
\begin{gather}\tag{MWOT}\label{MWOT}
	\inf_{\pi \in \Pi_M(\mu,\nu)} \int_{\R} C(x,\pi_x) \, \mu(dx),\\
\tag{MWOT'}\label{MWOT'}	
	\inf_{P \in \Lambda_M(\mu,\nu)} \int_{\R \times \mathcal P(\R)} C(x,p) \, P(dx,dp),	
\end{gather}
where $\Lambda_M(\mu,\nu)$ is the set of all $P\in \Lambda(\mu,\nu)$ giving full measure to 
$$\left\{(x,p) \in \R^d\times \mathcal P_1(\R^d) \colon x = \int_{\R^d} y \, p(dy)\right\},$$
 i.e., $P\in\Lambda_M(\mu,\nu)$ iff $P\in\Lambda(\mu,\nu)$ and
\begin{align*}
	\int_{\R^d \times \mathcal P(\R^d)} f(x,p) \, P(dx,dp) = 0\quad \forall f \in \tilde{\mathcal F}_M,
\end{align*}
where the set of martingale constraints $\tilde{\mathcal F}_M$ is given by
\begin{align*}
	\tilde{\mathcal F}_M := \Bigg\{f\in \mathcal C_b(\R^d\times \mathcal P_1(\R^d)) \colon   & \exists g \in \mathcal C_b(\mathcal P(\R^d)), h \in \mathcal C_b(\R^d)\\
	&\text{ s.t. } f(x,p) = g(p)h(x)\int_{\R^d}(x-y) \, p(dy) \Bigg\}.
\end{align*}

\begin{definition}\label{def:weak martingale monotonicity}~
\begin{enumerate}
	\item\label{it:wMM2} We call $\Gamma\subset \R^d\times \mathcal P_1(\R^d)$ martingale $C$-monotone iff for any $N\in\N$, any collection $(x_1,p_1),\ldots,(x_N,p_N)\in \Gamma$, and $q_1,\ldots,q_N\in \mathcal P_1(\R^d)$ such that $\sum_{i=1}^Np_i = \sum_{i=1}^N q_i$ and $\int_{\R^d} y \, p_i(dy) = \int_{\R^d} y \, q_i(dy)$, we have
	$$\sum_{i=1}^N C(x_i,p_i) \leq \sum_{i=1}^N C(x_i,q_i).$$
	\item\label{it:wMM3} A probability measure $P\in \mathcal P(\R^d\times\mathcal P_1(\R^d))$, which is supported on a martingale $C$-monotone set, is then called martingale $C$-monotone.
	\item\label{it:wMM4} A probability measure $\pi\in \mathcal P(\R^d\times\R^d)$ is called martingale $C$-monotone if $J(\pi)\in \mathcal P(\R^d\times\mathcal P_1(\R^d))$ is martingale $C$-monotone. (This is equivalent to Definition \ref{intro def:mart C-monotonicity}.) 
\end{enumerate}	
\end{definition}

Again, by \cite[Theorem 1.4]{BeGr14} we {show in the following theorem }that martingale $C$-monotonicity is a necessary optimality criterion.

\begin{theorem}\label{thm:MWOT monotonicity}
	Let $C\colon X \times \mathcal P_r(\R^d)\rightarrow {[0,\infty]}$ be measurable and $P^*\in \Lambda_M(\mu,\nu)$ optimal for \eqref{MWOT'} with finite value. Then $P^*$ is martingale $C$-monotone. { Moreover, if $C$ additionally satisfies} for all $x\in X$ and $Q \in \mathcal P_r(\mathcal P_r(\R^d))$
	\begin{align}\label{eq:integralconvexity2}
		C(x,I(Q)) \leq \int_{\mathcal P(\R^d)} C(x,p) \, Q(dp),
	\end{align}
	then any optimizer $\pi^*$ of \eqref{MWOT} with finite value is martingale $C$-monotone.
\end{theorem}

As before \eqref{eq:integralconvexity2} holds when $C(x,\cdot)$ is lower semicontinuous and convex, and in particular for $C(x,p)=\int_{\R^d} c(x,y) \, p(dy)$ {when $c\colon \R^d \times \R^d \to \R$ is lower semicontinuous and lower bounded.}

\begin{proof}
	Since \eqref{MWOT'} is an optimal transport problem under additional linear constraints, the first statement is a consequence of \cite[Theorem 1.4]{BeGr14}.
	To show the second assertion, we note that any martingale coupling $\pi\in \Pi_M(\mu,\nu)$ naturally induces an element in $\Lambda_M(\mu,\nu)$ by the embedding $J$, c.f. Section~\ref{sec:correct topology}. Let $P\in \Lambda_M(\mu,\nu)$, then $I(P_x)\,\mu(dx) \in \Pi_M(\mu,\nu)$ and by \eqref{eq:integralconvexity2}
	\begin{align*}
		\int_{\R^d \times \mathcal P(\R^d)} C(x,p)\,P(dx,dp) \geq \int_{\R^d \times \mathcal P(\R^d)} C(x,I(P_x))\,\mu(dx).
	\end{align*}
	Hence, $J(\pi^*)$ is optimal for \eqref{MWOT'}, {and we deduce from the first part} martingale $C$-monotonicity {of $J(\pi^\ast)$. Due to similar reasoning as in Remark \ref{rem:cycl monotonicity} \ref{rem:cycl monotonicity 1} we conclude that $\pi^\ast$ is also martingale $C$-montone}.
\end{proof}

From here on we assume that
$$d=1,$$
{but we hope that in the future a similar approach can be developed for higher dimensions.}

\begin{lemma}\label{lem:martingale approximative sequence}
	Let $N\in\N$ and $p_i \in \mathcal P_r(\R)$ with competitor $q_i \in \mathcal P_r(\R)$, $i = 1,\ldots,N$, i.e., 
	\[ \sum_{i=1}^N p_i = \sum_{i=1}^N q_i,\quad \int_\R y \, p_i(dy) = \int_\R y \, q_i(dy). \]
	{ Suppose there are sequences $\{p_1^k,\ldots,p_N^k\}$, $k \in \N$, of measures in $\mathcal P_r(\R)$ with $p_i^k \to p_i$ in $\mathcal W_r$.
	Then there exist approximative sequences $\{q_1^k,\ldots,q_N^k\}_{k \in \N}$ of competitors, i.e.,
	\[ \sum_{i = 1}^N p_i^k = \sum_{i = 1}^N q_i^k, \, {     \int_\R y \, p_i^k(dy) = \int_\R y \, q_i^k(dy)},\text{ and } q_i^k \to q_i\text{ in }\mathcal W_r.\]
	}
\end{lemma}

Since the proof of Lemma~\ref{lem:martingale approximative sequence} is slightly demanding, we first give for convenience of the reader a more concrete version of the argument in the simpler setting of $N=2$:

\begin{proof}[Proof of Lemma~\ref{lem:martingale approximative sequence} for $N=2$]
W.l.o.g. $q_1 \neq p_1$. Applying Lemma~\ref{lem:approximative sequence} we find a sequence $\{q_i^k\}_{k \in \N}$ which converges to $q_i$ in $\mathcal W_r$ {and such that $\sum_{i=1}^N q_i^k= \sum_{i=1}^N p_i^k$}.
	We may further assume $\int_\R y \, q_1^k(dy) < \int_\R y \, p_1^k(dy)$.
	We can decompose the measures $q_1,q_2,p_1,p_2$ into sub-probability measures $m_{i,j}$, $i,j\in\{1,2\}$ such that
	$$p_i = m_{i,1} + m_{i,2},\quad q_j = m_{1,j} + m_{2,j}.$$
	By equality of the mean values of $q_1$ and $p_1$, we find that
	$$\int_\R y \, m_{1,2}(dy) = \int_\R y \, m_{2,1}(dy).$$
	Thus, we find disjoint, open intervals $I_1,I_2$ with $\min(m_{1,2}(I_1),m_{2,1}(I_2)) > 0$ and $\sup(I_2) < \inf(I_1)$. Similarly, for each $k \in \N$ we can decompose $q_1^k,q_2^k,p_1^k,p_2^k$ in the same manner and obtain by the construction in Lemma~\ref{lem:approximative sequence} that $m_{i,j}^k$ converges to $m_{i,j}$ in $\mathcal W_r$. {When well-defined} denote by $\alpha_k > 0$ the constant such that
	\[\int_\R y\, q_1^k(dy) + \alpha_k \left(\frac{1}{m_{1,2}^k(I_1)} \int_{I_1} y \, m^k_{1,2}(dy) - \frac{1}{m^k_{2,1}(I_2)} \int_{I_2} y \, m^k_{2,1}(dy) \right)= \int_\R y \, p^k_1(dy).
	\]
	By $\mathcal W_r$-convergence, we have on the one hand $\liminf_k m_{1,2}^k(I_1) > 0$ and $\liminf_k m_{2,1}^k(I_2) > 0$, and on the other,
	$$\lim_{k \to \infty} \int_\R y q_1^k(dy) - \int_\R y p_1^k(dy) = 0,$$
	implying that {$\alpha_k$ is well-defined for $k$ sufficiently large, and} $\alpha_k\to 0$. 	Therefore, there is an index $k_0\in\N$ such that 
	\[\textstyle
	\tilde q_1^k = q_1^k + \alpha_k \left( \frac{\left.m^k_{1,2}\right|_{I_1}}{m^k_{1,2}(I_1)} - \frac{\left.m^k_{2,1}\right|_{I_2}}{m^k_{2,1}(I_2)} \right),\quad
	\tilde q_2^k = q_2^k - \alpha_k \left( \frac{\left.m^k_{1,2}\right|_{I_1}}{m^k_{1,2}(I_1)} - \frac{\left.m^k_{2,1}\right|_{I_2}}{m^k_{2,1}(I_2)} \right),
	\]
	are both probability measures for $k\geq k_0$, and $\{\tilde q_1^k,\tilde q_2^k\}_{k\geq k_0}$ has the desired properties.
\end{proof}

\begin{proof}[Proof of Lemma~\ref{lem:martingale approximative sequence} General Case]
	Let $s \in \R$, $F_p\colon\R\to[0,1]$ the cumulative distribution function of $p\in\mathcal P(\R)$, and define 
	$$I_s^1 := \{i \in \{1,\ldots,N\}\colon F_{p_i}(s) = 1\},\quad I_s^2 := \{ i \in \{1,\ldots,N\}\colon F_{q_i}(s) = 1\}.$$

	{ As a preparatory step, we show that
	\begin{equation}
		\label{eq:irred1}
		j \in \{1,\ldots,N\}\setminus I_s^1	\implies p_j((-\infty,s)) = 0,
	\end{equation}		
	already implies that $I_s^2 = I_s^1$ and
	\begin{equation}
		\label{eq:irred2}
		j \in \{1,\ldots,N\}\setminus I_2^1	\implies q_j((-\infty,s)) = 0.		
	\end{equation}
	This is achieved by observing the barycenters: assume \eqref{eq:irred1} and note that
	\begin{multline}
		\label{eq:irred3}
		\textstyle
		0 = \sum_{i \in I_s^1} \int_\R y \, p_i(dy) -  \int_\R y \, q_i(dy)
		\\ \textstyle = \int_\R y \, \left(  \sum_{i \in I_s^1} p_i - q_i \right)^+(dy) - \int_\R y \, \left(  \sum_{i \in I_s^1} q_i - p_i \right)^+(dy),
	\end{multline}
	where $(\cdot)^+$ of a signed measure denotes its positive part. Moreover, as $\sum_{i = 1}^N p_i = \sum_{i = 1}^N q_i$ we obtain
	\[
		\sum_{i \in I_s^1} q_i \vert_{(-\infty,s)} \leq \sum_{i = 1}^N p_i\vert_{(-\infty,s)} \leq \sum_{i \in I_s^1} p_i,
	\]
	and consequently $(\sum_{i \in I_s^1} p_i - q_i)^+$ is concentrated on $(-\infty,s]$, whereas $(\sum_{i \in I_s^1} q_i - p_i)^+$ is concentrated on $[s,+\infty)$.
	From \eqref{eq:irred3} follows that $\left(\sum_{i \in I_s^1} p_i - q_i\right)^+ = \left(\sum_{i \in I_s^1} q_i - p_i\right)^+=0$, thus $\sum_{i \in I_s^1} p_i = \sum_{i \in I_s^1} q_i$ and $I_s^1 \subseteq I_s^2$.
	Assume that there is $j \in I_s^2 \setminus I_s^1$.
	Then
	\[
		\sum_{i = 1}^N q_i \vert_{(-\infty,s)} = \sum_{i = 1}^N p_i\vert_{(-\infty,s)} = \sum_{i \in I_s^1} p_i\vert_{(-\infty,s)} = \sum_{i \in I_s^1} q_i\vert_{(-\infty,s)}
	\]
	shows that $q_j((-\infty,s)) = 0$. Since $j \in I_s^2$ we have that $q_j((s,+\infty)) = 0$, and 
	consequently $q_j = \delta_s$. The barycenter of $p_j$ coincides with $s$, but from $j \notin I_s^1$
	we deduce $p_j((s,+\infty)) > 0$ and $p_j((-\infty,s)) > 0$, which violates \eqref{eq:irred1}.	
	Therefore $I_s^1 = I_s^2$, and \eqref{eq:irred2} is satisfied.

	Clearly, if there exists $s \in \R$ such that either \eqref{eq:irred1} or \eqref{eq:irred2} holds, then by the preparatory step we get $I_s^1=I_s^2=:I_s$ and
	$\sum_{i \in I_s} p_i = \sum_{i \in I_s} q_i,$ and 
	\[
		j \in \{1,\ldots,N\}\setminus I_s \implies p_j((-\infty,s)) = 0,\quad j\in \{1\ldots,N\}\setminus I_s \implies q_j((-\infty,s)) = 0.
	\]
	In this case we can split the problem into two parts: Finding sequences of competitors for the index sets $I_s$ and $\{1,\ldots,N\} \setminus I_s$. It is sufficient to show the existence of such a sequence for the sub problem $I_s$, where we also assume that $s$ is minimal.
		
	Thus, assume without loss of generality that $s$ is minimal and $I_s^1 = I_s^2 = \{1,\ldots,N\}$, $N > 1$.
	Denote the convex hull of the support of $q_i$ by
	\[ S_i := \text{co}(\supp q_i)\quad i = 1,\ldots,N. \]	
	We can also assume without loss of generality that $(p_i,q_j)$ are pairwise different for $i,j \in \{1,\ldots,N\}$.

	Let $i_{\min},i_{\max} \in \{1,\ldots,N\}$ be such that $\inf S_{i_{\min}}$ is minimal and $\sup S_{i_{\max}}$ is maximal. Under these assumptions, we get that $p_{i_{\min}},q_{i_{\min}}, p_{i_{\max}}$, and $q_{i_{\max}}$ cannot be concentrated on a single point, thus,
	\begin{equation}
		\label{eq:minmax pos measure}
		\lambda(S_{i_{\min}}) > 0 \text{ and }\lambda(S_{i_{\max}}) > 0,
	\end{equation}
where $\lambda$ is the Lebesgue measure.

	Applying now Lemma~\ref{lem:approximative sequence} we find for every $i \in \{1,\ldots,N\}$ a sequence $\{q_i^k\}_{k \in \N}$ with $\sum_{i = 1}^N q_i^k = \sum_{i = 1}^N p_i^k$ and which converges to $q_i$ in $\mathcal W_r$, in particular,
	\begin{equation}
		\label{eq:barys converge}
		\lim_{k \to \infty} \left|\int_\R y p_i^k(dy) - \int_\R y q_i^k(dy)\right| = 0.
	\end{equation}

	Consider the following case. Let $i,j \in \{1,\ldots,N\}$ such that one of the following is true: either
	\begin{equation}
		\label{eq:cases}
		\lambda(S_j \cap S_i)>0 \quad \text{or}\quad S_j \subset \text{int}(S_i),
	\end{equation}
	where $\text{int}(S_i)$ denotes the interior of $S_i$. 
	Then there are open intervals $O_{i,j}^+, O_{i,j}^-, O_{j,i}^+, O_{j,i}^-$ such that
	\begin{gather*}
		q_i(O_{i,j}^t) > 0,\quad q_j(O_{j,i}^t) > 0,\quad t \in \{-,+\},\\
		\sup O_{i,j}^+ < \inf O_{j,i}^- ,\quad \sup O_{j,i}^+ < \inf O_{i,j}^-.
	\end{gather*}
	By weak convergence of $q_i^k$ to $q_i$, the Portmanteau theorem implies 
	\begin{align}
		\label{eq:portmanteau mass}
		\begin{split}
		\liminf_{k \to \infty} q_i^k(O_{i,j}^+) > 0, &\quad \liminf_{k\to \infty} q_i^k(O_{i,j}^-) > 0, \\
		\liminf_{k \to \infty} q_j^k(O_{j,i}^+) > 0, &\quad \liminf_{k\to \infty} q_j^k(O_{j,i}^-) > 0.
		\end{split}
	\end{align}		
	In particular, when $k$ is sufficiently large we can by \eqref{eq:barys converge} and \eqref{eq:portmanteau mass} replace $q_i^k$ and $q_j^k$ with the probability measures
	\begin{align*}
		\tilde q_j^k := q_j^k + \alpha_+^k \left( \frac{q_i^k \vert_{O_{i,j}^-}}{q_i^k(O_{i,j}^-)} - \frac{q_j^k\vert_{O_{j,i}^+}}{q_j^k(O_{j,i}^+)} \right) + \alpha_-^k \left( \frac{q_i^k \vert_{O_{i,j}^+}}{q_i^k(O_{i,j}^+)} - \frac{q_j^k\vert_{O_{j,i}^-}}{q_j^k(O_{j,i}^-)} \right),
	\end{align*}
	 and $\tilde q_i^k := q_i^k - \tilde q_j^k + q_j^k$, where $\{\alpha_-^k\}_{k \in \N}$ and $\{\alpha_+^k\}_{k \in \N}$ are non-negative sequences converging to zero, and $\tilde q_i^k$ and $p_i^k$ have the same
	barycenter (i.e.\ if the barycenter of $q_j^k$ is smaller than that of $p_j^k$ we set $\alpha_+^k>0$ and $\alpha_-^k=0$, etc.). Thus, $\{\tilde q^k_i\}_{k \in \N}$ and $\{\tilde q^k_j\}_{k \in \N}$ converge in $\mathcal W_r$ to
	$q_i$ and $q_j$, respectively.

	We will call $i \in \{1,\ldots, N\}$ a \emph{pivot}, if for all $\hat i \in \{1,\ldots,N\}$ with
	$\sup S_{\hat i} \leq \inf S_i$, the barycenters of $q_{\hat i}^k$ are correct (when $k$ is sufficiently large), that means, there is $k_0 \in \N$ such that
	\begin{equation}
		\label{eq:barycenters correct}
		k \geq k_0 \implies \int_\R y \, q_{\hat i}^k(dy) = \int_\R y \, p_{\hat i}^k(dy).
	\end{equation}
	For such a pivot $i$, consider all $j \neq i$ with \eqref{eq:cases}.
	By the previous step in this proof, there exists $k_1 \in \N$ such that we can change $q_i^k$ and $q_j^k$ for $k \geq k_1 \in \N$ (and denote the probability measures  $\tilde q_i^k$ and $\tilde q_j^k$ for notational convenience again by $q_i^k$ and $q_j^k$ respectively) such that
	\begin{equation}
		\label{eq:barycenters correct'}
		k\geq k_1 \implies \int_\R y \, q_j^k(dy) = \int_\R y \, p_j^k(dy).
	\end{equation}
	There are two possible cases:
	\begin{enumerate}
	\item \eqref{eq:cases} held true for all $j \neq i$ with  $\inf S_i < \sup S_j$, which would imply that (given $k$ is sufficiently large), we have not only corrected the barycenter of $q_j^k$ but also the one of $q_i^k$, since $$\int_\R yq_i^k(dy) = \sum_{l=1}^N \int_\R y p_l^k(dy) - \sum_{j\neq i}\int_\R y q_j^k(dy) = \int_\R y p_i^k(dy) .$$
	Hence, we have found the desired sequences.
	\item There are indices $k \in \{1,\ldots,N\}$ with $\sup S_i \leq \inf S_k \leq \sup S_k$. Due to minimality of $s$, there is an index $l \in \{1,\ldots,N\}$ with $\inf S_l < \sup S_i < \sup S_l$. Hence \eqref{eq:cases} holds and we can use $\{q_l^k\}_{k \in \N}$ to fix the barycenters of $\{q_i^k\}_{k \in \N}$ for $k$ sufficiently large.
	Moreover, $l$ is a pivot: Let $\hat l \in \{1,\ldots,N\}$ with $\sup S_{\hat l} \leq \inf S_l$, then
	either $\sup S_{\hat l} \leq \inf S_{i}$ whereby \eqref{eq:barycenters correct} holds for $\hat l$,
	or $\inf S_i < \sup S_{\hat l} \leq \inf S_l < \sup S_i$.
	In the latter case, we have that $\hat l$ satiesfies \eqref{eq:cases} and therefore \eqref{eq:barycenters correct'}. Notice that \begin{equation}
\sup	S_l > \sup S_i. \label{eq:algo_advances}
	\end{equation} 
	\end{enumerate}
	
	Taking $i=i_{\min}$ as our initial pivot, we can iterate the above reasoning. Since there are only finitely many elements, the procedure terminates as ensured by \eqref{eq:algo_advances}.
	}
\end{proof}

The key ingredient of this part is the following stability result concerning the notion of martingale $C$-monotonicity:

\begin{theorem}\label{thm:stability MWOT}
	Let $C,C_k\in \mathcal C(\R\times \mathcal P_r(\R))$, $k\in\N$, and $C_k$ converges uniformly to $C$. If $P,P^k\in \mathcal P_r(\R\times\mathcal P_r(\R))$, $k\in\N$, are such that
	\begin{enumerate}[label=(\alph*)]
		\item for all $k\in\N$ the measure $P^k$ is martingale $C_k$-monotone,
		\item the sequence $\{P^k\}_{k\in\N}$ converges to $P$,
	\end{enumerate}
then $P$ is martingale $C$-monotone. Moreover, if $\pi,\pi^k\in \mathcal P_r(\R\times \R)$ and $C_k$ is convex in the second argument and such that
	\begin{enumerate}[label=(\alph*')]
		\item for all $k\in\N$ the measure $\pi^k$ is martingale $C_k$-monotone,
		\item the sequence $\{\pi^k\}_{k\in\N}$ converges to $\pi$,
	\end{enumerate}
	then $\pi$ is martingale $C$-monotone.
\end{theorem}

\begin{proof}
	The proof runs parallel to the one of Theorem~\ref{thm:stability WOT}: Using Lemma~\ref{lem:martingale approximative sequence} we can alter Lemma~\ref{lem:closed Gamma} such that for all $\epsilon \geq 0$ and $N\in\N$ the set
	\begin{align*}
	\tilde{\Gamma}^\epsilon_N := \Bigg\{(x_i,p_i)_{i=1}^N \in (\R \times \mathcal P_r(\R))^N\Big|\forall m_1,\ldots,m_N\in P_r(\R)\text{ s.t. }\sum_{i=1}^N p_i = \sum_{i=1}^N m_i \text{ and }\\
	\int_\R y \, m_i(dy) = \int_\R y \, p_i(dy) \text{ for } i = 1,\ldots, N, \text{ we have } \sum_{i=1}^NC(x_i,p_i)\leq \sum_{i=1}^N C(x_i,m_i) + \epsilon\Bigg\}
	\end{align*}
	is closed. The aim is to construct a martingale $C$-monotone set $\Gamma$ on which $P$ is concentrated. So, we write $P^{k,\otimes N}$ and $P^{\otimes N}$ for the $N$-fold product measure of $P^k$ resp.\ $P$ where $N\in\N$. By martingale $C_k$-monotonicity and uniform convergence we find for any $\epsilon > 0$ a natural number $k_0$ such that $P^{k,\otimes N}$, $k\geq k_0$, is concentrated on $\tilde \Gamma^\epsilon_N$. { As $\tilde \Gamma_N^\epsilon$ is closed}, the Portmanteau theorem yields that $P^{\otimes N}$ is concentrated on $\tilde \Gamma_N^\epsilon$:
	\begin{align*}
	1 = \limsup_k P^{k,\otimes N}(\tilde \Gamma^\epsilon_N) \leq P^{\otimes N}(\tilde \Gamma^\epsilon_N) = 1.
	\end{align*}
	As a consequence, we find that $P^{\otimes N}$ gives full measure to the closed set $\tilde \Gamma_N := \tilde \Gamma^0_N$.
	{With the same line of argument as in the proof of Theorem \ref{thm:stability WOT}, we can find from here a closed, martingale $C$-monotone set $\tilde \Gamma$ with $P(\tilde \Gamma) = 1$.
	}
	%
	
	To show the second assertion, we embed $\pi^k\in\mathcal P_r(\R \times \R)$ into $\mathcal P_r(\R \times \mathcal P_r(\R))$ owing to the map $J$. Then, by compactness of $\Lambda_M(\mu,\nu)$, we find an accumulation point $P\in \mathcal P_r(\R\times \mathcal P_r(\R))$ of $\{J(\pi^k)\}_{k\in\N}$. Note that $P$ gives full measure to $\{(x,p) \in \R \times \mathcal P_1(\R)\colon x = \int_\R y \, p(dy)\}$, and
	\[
		\mu(dx) \, I(P_x) =: \pi \in \Pi_M(\mu,\nu)
	\]
	determines a martingale coupling, which is likewise a {$\mathcal W_r$-}accumulation point of $\{\pi^k\}_{k\in\N}$. 
	{Since $P \in \mathcal P_r(\R \times \mathcal P_r(\R))$ we have $\mu$-almost surely that
	$I(P_x) \in \mathcal P_r(\R)$.
	As $P$ is concentrated on the martingale $C$-monotone set $\tilde \Gamma$, we find for any $x\in X$ such that $P_x(\tilde \Gamma_x) = 1$ and $I(P_x) \in \mathcal P_r(\R)$} a sequence of measures $p_i^x \in \tilde \Gamma_x \subset \mathcal P_r(\R),~i\in\N$, with
	\[
		q_n^x := \frac{1}{n}\sum_{i=1}^n p_i^x \to I(P_x) = \pi_x,\quad n\to \infty,\quad  \text{ in }\mathcal W_r.	
	\]
	{Since $C$ is convex in the second argument, we can reason as in Remark \ref{rem:cycl monotonicity} \ref{rem:cycl monotonicity 2}, and find that for $\mu$-a.e.\ $x$} $(x,q_n^x)$ is already contained in the martingale $C$-monotone set $\tilde \Gamma$. 
	Then by closure of $\tilde \Gamma$ we conclude $(x,\pi_x) \in \tilde \Gamma$ for $\mu$-a.e.\ $x$, and {by the same reasoning as in Remark \ref{rem:cycl monotonicity} \ref{rem:cycl monotonicity 1}} martingale $C$-monotonicity of $\pi$.
\end{proof}

\begin{lemma}\label{lem:c,mart mon equiv C,mart mon}
	{Let $\pi \in \Pi_M(\mu,\nu)$, $b \in L^1(\nu)$, $c \colon X \times Y \to \R$ be jointly measurable, and $c(x,\cdot)$ be upper semicontinuous and upper bounded by a postive multiple of $b$ for all $x\in\R$.}
	Then $\pi$ is $(c,\mathcal F_M)$-monotone if and only if $\pi$ is martingale $C$-monotone (with $C(x,p) := \int_\R c(x,y)p(dy)$).
\end{lemma}
\begin{proof}
Let $\hat \Gamma\subset \R\times\R$ be $(c,\mathcal F_M)$-monotone with $\pi(\hat\Gamma)=1$. Consider the Borel measurable set{
	\[
		\Gamma := \left\{(x,p) \in X\times \mathcal P_1(Y)\colon p(\hat \Gamma_x) = 1, \int_\R y \, p(dy) = x, \int_\R |b(y)| \, p(dy) < \infty \right\},
	\]
	where $(x,\pi_x) \in \Gamma$ for $\mu$-almost every $x \in X$.}
	Take any sequence $(x_1,p_1),\ldots,(x_N,p_N)\in \Gamma$ with competitors $q_1,\ldots q_N\in\mathcal P_1(\R)$, i.e.,
	\[
		\sum_{i=1}^N p_i = \sum_{i=1}^N q_i,\quad \int_\R y \, p_i(dy) = \int_\R y \, q_i(dy).	
	\]
	We find for any $(x_i,p_i) \in \Gamma$, a sequence of finitely supported measures $\{p_i^k\}_{k\in\N}$ where {for all $k\in\N$ we have that $p_i^k(\hat \Gamma_{x_i}) = 1$, $\int_\R y \, p_i^k(dy) = x_i$, and }
	\begin{align}
		\label{eq:cmart cost converges}
		\lim_{k \to \infty}\int_\R c(x_i,y) \, p_i^k(dy) = \int_\R c(x_i,y) \, p_i(dy),
	\\	\label{eq:cmart majorant converges}
	{\lim_{k\to\infty}\int_\R |b(y)| \, p_i^k(dy) = \int_\R|b(y)| \, p(dy).}
	\end{align}{
	Thus, Lemma~\ref{lem:martingale approximative sequence} provides $\{ q_1^k,\ldots,q_N^k \}$, sequences of feasible and finitely supported competitors  with corresponding limit points $\{ q_1,\ldots,q_N \}$.}
	Then $(c,\mathcal F_M)$-monotonicity yields
	\begin{align*}
		\sum_{i=1}^N \int_\R c(x_i,y) \, p_i(dy) &= \lim_{k \to \infty}\sum_{i=1}^N\int_\R c(x_i,y) \, p_i^k(dy) \\
		&\leq \limsup_{k \to \infty} \sum_{i=1}^N \int_\R c(x_i,y) \, q_i^k(dy) \leq \sum_{i=1}^N \int_\R c(x_i,y) \, q_i(dy),
	\end{align*}{
	where we used $\eqref{eq:cmart cost converges}$ for the first equality, and upper semicontinuity  of $c(x,\cdot)$, upper boundedness of $c(x,\cdot) - a(x) b(\cdot)$ for some $a(x) > 0$, and \eqref{eq:cmart majorant converges}} in the last inequality.
	{We have shown that $\Gamma$ is martingale $C$-monotone. 
	Since $(x,\pi_x) \in \Gamma$ for $\mu$-a.e.\ $x \in \R$, we conclude that $\pi$ is martingale $C$-monotone.}

	Now, let $\pi$ be martingale $C$-monotone on a Borel measurable set $\Gamma\subset \R\times\mathcal P(\R)$,
	{that is, $J(\pi)(\Gamma) = 1$.} Due to the variation of Lusin's theorem, Lemma~\ref{lem:Gammahat}, there is an analytically measurable set $\hat\Gamma\subset X\times Y$ satisfying:
		\begin{itemize}
	\item[(i)] For any $(x,p) \in \Gamma$ we have that $p$ is concentrated on the fibre $\hat \Gamma_x = \{y \in Y \colon (x,y) \in \hat \Gamma\}$, i.e., $p(\hat\Gamma_x) = 1$.
	\item[(ii)] For any $(x,y) \in \hat\Gamma$ we find $(x,p) \in \Gamma$ and a Borel measurable set $K \subset \hat \Gamma_x$ such that
	\begin{enumerate}
		\item $c$ restricted to the fibre $\{x\} \times K$ is continuous,
		\item  $y\in \supp(p|_K)\cap K$ and
			\begin{equation}
				\label{eq:limit to cost}
				\int_{B_\delta(y) \cap K} \frac{c(x,z)}{p(B_\delta(y)\cap K)} \, p(dz) \rightarrow c(x,y)\quad\text{for }\delta\searrow 0.
			\end{equation}
	\end{enumerate}
	\end{itemize}
	 In particular, we have that $\pi$ is concentrated on $\hat \Gamma$ by property $(i)$. To see that $\hat \Gamma$ is $(c,\mathcal F_M)$-monotone, take a finite subset of $\hat \Gamma$, i.e., $G := \{(x_1,y_1),\ldots,(x_N,y_N)\} \subset \hat \Gamma$. Let $\alpha$ be supported on $G$ and $\beta$ be a competitor, i.e.,
	\begin{gather*}
		 \proj_1(\alpha) =  \proj_1(\beta),\quad \proj_2(\alpha) =  \proj_2(\beta),\\\int_\R y \, \alpha_{x_i}(dy) = \int_\R y \, \beta_{x_i}(dy),\quad i=1,\ldots,N.
	\end{gather*}
	 By property $(ii)$ of $\hat \Gamma$ we find for each $(x_i,y_i)$, $1\leq i \leq N$, {$(x_i,p_i) \in \Gamma$ and sets $K_i$ such that $c$ is continuous on $\{x_i\}\times K_i$, $y_i \in K_i$, and \eqref{eq:limit to cost} holds.}
	 Let
	 \[
		 \alpha^k(dx,dy) := \sum_{i=1}^N \delta_{x_i}(dx)\frac{\alpha^k(x_i,y_i)}{p_i(K_i\cap B_{\frac{1}{k}(y_i)})} p_i|_{K_i\cap B_{\frac{1}{k}}(y_i)}(dy),\quad k\in\N.
	\]
	Then $\alpha^k$ converges weakly to $\alpha$, and by Lemma~\ref{lem:martingale approximative sequence} we find a sequence { $\{\beta^k\}_{k \in \N}$, where $\beta^k$ is a competitor of $\alpha^k$}, which converges weakly to $\beta$.
	{As $c(x,\cdot)$ is upper semicontinuous and finitely valued, there exists $k_0 \in \N$ with 
	\begin{equation}
		\label{eq:cmart usc=>bounded}
		\sup \left\{ c(x,y)	\colon y \in \bigcup_{i = 1}^N B_{\frac{1}{k_0}}(y_i) \right\} < \infty.
	\end{equation}}	
	Thus, we have
	\begin{align*}
		\int_{\R \times \R} c(x,y) \, \alpha(dx,dy) &= \lim_{k \to \infty} \int_{\R \times \R} c(x,y) \, \alpha^k(dx,dy) 
	\\	&\leq \liminf_{k \to \infty} \int_{\R \times \R} c(x,y) \, \beta^k(dx,dy) \leq \int_{\R \times \R} c(x,y) \, \beta(dx,dy),
	\end{align*}
	where we {obtain the first equality due to \eqref{eq:limit to cost}, the first inequality due to martingale $C$-monotonicity, and the final inequality due to upper semicontinuity and \eqref{eq:cmart usc=>bounded}.}
\end{proof}



\begin{proof}[Proof of Theorem \ref{Thm stab mot intro}]
	{Since $\pi^k$ is optimizer of \eqref{MOT} for cost $c_k$ with marginals $\mu_k$ and $\nu_k$,
	$\pi^k$ is $(c_k,\mathcal F_M)$-monotone by \cite[Theorem 1.4]{BeGr14}.}
	By Lemma~\ref{lem:c,mart mon equiv C,mart mon}, we find that $\pi^k$ is martingale $C_k$-monotone.
	Then Theorem~\ref{thm:stability MWOT} shows that martingale monotonicity is preserved for $k\to \infty$. 
	Hence, $\pi$ is martingale $C$-monotone {and thus $(c,\mathcal F_M)$-monotone by Lemma \ref{lem:c,mart mon equiv C,mart mon}.}
	Finally, $\pi$ is optimal for \eqref{MOT} with cost $c$ by \cite[Theorem 1.3]{Gr16}.
\end{proof}

\begin{proof}[Proof of Corollary \ref{Cor stab mot intro}]
Let $\pi^k$ be optimal for \eqref{MOT} for the cost function $c_k$ and marginal measures $\mu_k,\nu_k$. We may apply Theorem \ref{Thm stab mot intro} showing that every accumulation point (with respect to weak convergence) of $\{\pi^k\}_{k \in \N}$ is an optimizer for \eqref{MOT} for the cost function $c$ and marginal measures $\mu,\nu$.
{On $\R^2$ we may choose the metric $D((x,y),(\bar x,\bar y))= \sqrt[r]{|x-\bar x|^r+|y-\bar y|^r}$ in order to define the $r$-Wasserstein metric on $\mathcal P_r(\mathbb R^2)$.
Then
\begin{multline*}
	\int_{\R\times\R} D((0,0),(x,y))^r\,\pi^k(dx,dy) = \int_\R |x|^r \, \mu_k(dx) + \int_\R |y|^r \, \nu_k(dy)
\\	\to \int_\R |x|^r \, \mu(dx) +\int_\R |y|^r \, \nu(dy) = \int_{\R\times\R} D((0,0),(x,y)) \, \pi(dx,dy),
\end{multline*}
as $k$ tends to $\infty$}
for any coupling $\pi$ with marginals $\mu,\nu$.
It follows {by \cite[Definition 6.8 (i)]{Vi09}} that the accumulation points of $\{\pi^k\}_{k \in \N}$ under the weak topology or under the $\mathcal W_r$-topology coincide.
The following inequality is immediate
\[
	\liminf_{k\to\infty}\, \int_{\R\times\R} c_k(x,y) \, \pi^k(dx,dy)\geq\inf_{\pi\in\Pi_M(\mu,\nu)} \int_{\R\times\R} c(x,y)\, \pi(dx,dy).
\]
In order to  prove
\[
	\limsup_{k\to\infty}\, \int_{\R\times\R} c_k(x,y) \, \pi^k(dx,dy)\leq\inf_{\pi\in\Pi_M(\mu,\nu)} \int_{\R\times\R} c(x,y) \, \pi(dx,dy),
\]
it suffices to observe that if (for some subsequence which we do not track) $\pi^k\to \pi$ in $\mathcal W_r$, then $\int_{\R \times \R} c_k(x,y) \pi^k(dx,dy)\to\int_{R \times \R} c(x,y) \, \pi(dx,dy)$, since $\pi$ must be optimal for the r.h.s.
{But this is clear since $c_k\to c$ uniformly and $c$ is dominated by a positive multiple of $(x,y) \mapsto 1 + D((0,0),(x,y))^r$.}
\end{proof}

\section{The relation of OT and WOT}\label{sec:wotot}

In this part we explore the relationship between (classical) $c$-cyclical monotonicity in optimal transport, and $C$-monotonicity in optimal weak transport, in the case when $C(x,p)=\int_Y c(x,y)p(dy)$.
{We recall the notion of $c$-cyclical monotonicity, cf.\ \cite[Definition 5.1]{Vi09}:
\begin{definition}[{$c$}-cyclical monotonicity]
	Let $c \colon X \times Y \to \mathbb R$.
	\begin{enumerate}
		\item A set $\Gamma \subset X \times Y$ is called $c$-cyclically monotone iff for any
		$N \in \N$, and any collection $(x_1,y_1),\ldots,(x_N,y_N) \in \Gamma$, we have
		\[
			\sum_{i = 1}^N c(x_i,y_i) \leq \sum_{i = 1}^N c(x_i,y_{i+1}),	
		\]
		with the convention $y_{N+1} = y_1$.
		\item A probability measure $\pi \in \mathcal P(X \times Y)$, which is concentrated on a $c$-cyclically monotone set, is then called $c$-cyclically monotone.
	\end{enumerate}
\end{definition}}

Our main result {of this section} is: 

\begin{theorem}\label{thm:C-Mon equiv c-cycl}
	{Let $\pi \in \Pi(\mu,\nu)$, $b \in L^1(\nu)$, $c\colon X\times Y\to \R$ be jointly measurable, and $c(x,\cdot)$ be upper semicontinuous and upper bounded by a positive multiple of $b$ for all $x\in X$.
	Then $\pi$ is $c$-cyclically monotone if and only if $\pi$ is $C$-monotone (with $C(x,p) := \int_Y c(x,y) \, p(dy)$).}
\end{theorem}


\begin{lemma}\label{lem:restriction property}
	{Let $c\colon X \times Y \to \R$ be jointly measurable. Let $\pi \in \Pi(\mu,\nu)$ be $C$-monotone with $C(x,p) := \int_Y c(x,y)\,p(dy)$ and $c(x,\cdot) \in L^1(\pi_x)$ for $\mu$-almost every $x \in X$. Then}
	for all $A\in\mathcal B(X\times Y)$ the restriction of $\pi$ to $A$, i.e., $\pi|_A$, is also $C$-monotone.
\end{lemma}

\begin{proof}
	Let $x_1,\ldots,x_N\in\Gamma \cap \{x\in X\colon \pi_{x}(A_{x}) > 0\}$, $N\in\N$, where $\Gamma$ is a $C$-monotone subset and $A_x := \{y\in Y \colon (x,y) \in A\}$. To shorten notation we write
	\begin{align*}
		p_i := \pi_{x_i}|_{A_{x_i}}, \quad \bar p_i := \frac{1}{p_i(A_{x_i})} p_i,\quad i=1,\ldots,N.
	\end{align*}
	Without loss of generality we can assume that restricting and disintegrating commutes, $\bar p_i = (\pi|_A)_{x_i}$ for all $i = 1,\ldots,N$.
	{
	Let $n > 1$ be a natural number with reciprocal value smaller than $\min_i p_i(A_{x_i})$.
	Define the probability measures
	}
	\[ r_i = \frac{n\pi_{x_i} - \bar p_i}{n-1} \in \mathcal P(Y),\quad i = 1,\ldots,N,\]
	{which by linearity of $p \mapsto C(x,p)$ satisfy}
	\begin{equation}
		\label{eq:identity}
		n \sum_{i=1}^N C(x_i,\pi_{x_i}) = \sum_{i=1}^N C(x_i,\bar p_i) + (n-1) C(x_i,r_i).
	\end{equation}
	To show $C$-monotonicity of $\pi|_A$, let $\bar q_1,\ldots,\bar q_N\in \mathcal P(Y)$ with $\sum_{i=1}^N\bar q_i = \sum_{i=1}^N \bar p_i$. Define
	\begin{align*}
		q_{i,j} := \begin{cases} \bar q_i & j = 1,\\ r_i & 2\leq j \leq n,\end{cases}\quad i = 1,\ldots,N,
\end{align*}
{whence, $q_i := \frac{1}{n}\sum_{j = 1}^n q_{i,j} \in \mathcal P(Y)$, $i = 1,\ldots,N$, define competitors of $(\pi_{x_i})_i$, since
\[
	\sum_{i = 1}^N q_i = \frac{1}{n}\sum_{i=1}^N \sum_{j=1}^n q_{i,j} = \sum_{i=1}^N \pi_{x_i}.
\]}
From $C$-monotonicity of $\pi$ we derive that{
\[
	n \sum_{i=1}^N C(x_i,\pi_{x_i}) \leq n \sum_{i = 1}^N C(x_i,q_i) = \sum_{i=1}^N \sum_{i=1}^n C(x_i,q_{i,j}),
\]
which is by \eqref{eq:identity}} equivalent to
\[
	\sum_{i=1}^NC(x_i,\bar p_i) \leq \sum_{i=1}^N C(x_i,\bar q_i).
\]\end{proof}

The next proposition has Theorem~\ref{thm:C-Mon equiv c-cycl} as an immediate corollary:
\begin{proposition}\label{prop:connection}
	Let $c\colon X \times Y \rightarrow \R$ be jointly measurable, {$b \colon Y \to \R$} be measurable, and {$c(x,\cdot)$ be upper semicontinuous and dominated by a positive multiple of $b$ for each $x \in X$.}
	If $\Gamma \subset X \times \mathcal P(Y)$ is a $C$-monotone analytic set (where $C(x,p) := \int_Y c(x,y)p(dy)$), then the set $\hat \Gamma \subset X \times Y$ from Lemma~\ref{lem:Gammahat} is $c$-cyclically monotone. Conversely, if a $c$-cyclical monotone set $\hat\Gamma\subset X\times Y$ is given, then
	\[
		\Gamma := \left\{(x,p) \in X\times \mathcal P(Y)\colon p(\hat\Gamma_x) = 1,{b} \in L^1(p) \right\}
	\]
	is $C$-monotone.
\end{proposition}

\begin{proof}
	Suppose that $\Gamma$ is $C$-monotone. Let a finite number of points $(x_1,y_1),\ldots,(x_N,y_N)$ in $\hat \Gamma$ be given. We find $(x_i,p_i) \in \Gamma$ with $p_i(\hat \Gamma_{x_i}) = 1$ as in Lemma~\ref{lem:Gammahat}. This means that for each $y_i$ there is a Borel measurable set $K_i\subset Y$ with $y_i \in \supp(p_i|_{K_i})$, $c|_{\{x_i\}\times K_i}$ is continuous, and
	\[
		p_i^\epsilon := \frac{1}{p_i(B_\epsilon(y_i) \cap K_i)}p_i|_{B_\epsilon(y_i)\cap K_i},\quad
		\int_{K_i\cap B_\epsilon(y_i)} \frac{c(x_i,z)}{p(K_i\cap B_\epsilon(y_i))} \, p(dz) \rightarrow c(x_i,y_i).
	\]
	The measure $\frac{1}{N}\sum_{i = 1}^N \delta_{x_i}p_i$ is by construction $C$-monotone. 
	By the restriction property of Lemma~\ref{lem:restriction property}, the measure $\frac{1}{N}\sum_{i = 1}^N \delta_{x_i}p_i^\epsilon$ is likewise $C$-monotone.
	Then by upper semicontinuity {of $c(x,\cdot)$ and \eqref{eq:cmart usc=>bounded} (which is here applicable),} we conclude (with the convention that $N+1 = 1$):
	\begin{align*}
		\sum_{i=1}^N c(x_i,y_i) = \lim_{\epsilon\searrow 0} \sum_{i=1}^N C(x_i,p^\epsilon_i)
		\leq \limsup_{\epsilon\searrow 0} \sum_{i=1}^N C(x_i,p^\epsilon_{i+1}) \leq \sum_{i=1}^N c(x_i,y_{i+1}).
	\end{align*}
	Now let $\hat \Gamma$ be a $c$-cyclical monotone set.{	
	Recall that $c(x,\cdot)$ is dominated by $b$ and $b \in L^1(p)$ for all $(x,p) \in \Gamma$.
	Therefore, by the law of large numbers, we have for an iid sequence $(Y_i)_{i\in \N}$ distributed according to $p$ that almost surely
	\[
		\lim_{n \to \infty}\frac{1}{n} \sum_{i=1}^n c(x,Y_i) = \int_Y c(x,y) \, p(dy),\quad
		\lim_{n \to \infty}\frac{1}{n} \sum_{i = 1}^n |b(Y_i)| = \int_Y |b(y)| \, p(dy).
	\]}
	Thus, we can approximate $p$ by discrete measures concentrated on $\hat\Gamma$ and  thereby obtain $C$-monotonicity of $\Gamma$: Let $(x_1,p_1),\ldots,(x_N,p_N)\in \Gamma$ and $q_1,\ldots,q_N\in\mathcal P(Y)$ with $\sum_{i=1}^N p_i = \sum_{i=1}^N q_i$. We find by discretely approximating $p_1,\ldots,p_N$ on $\hat \Gamma$, and then using $c$-cyclical monotonicity, the sequence of competitors constructed in Lemma~\ref{lem:approximative sequence}, {and upper semicontinuity of $c(x,\cdot)$ and upper boundedness of $c(x,\cdot) - a(x)b(\cdot)$ for some $a(x) > 0$} that
	\[
		\sum_{i=1}^N C(x_i,p_i) = \lim_{n \to \infty} \frac{1}{n} \sum_{i=1}^N \sum_{j=1}^n c(x_i,y_j^i) \leq \limsup_n \sum_{i=1}^N \int_Y c(x_i,y) \, q_i^n(dy) \leq \sum_{i=1}^N C(x_i,q_i).
	\]
\end{proof}

\section{Epilogue}\label{sec:epi}

Convexity is a natural assumption in the setting of weak transport. It is known that convexity of $C(x,\cdot)$ is necessary to obtain general existence of minimizers in the space of couplings, see \cite[Example 3.2]{AlBoCh18}. Similarly, convexity is required for $C$-monotonicity to be a necessary optimality criterion:

\begin{example}\label{ex:1}
	Let $X = \{0\}$, $Y = \{0,1\}$, $\mu = \delta_0$, $\nu = \frac{1}{2}(\delta_0 + \delta_1)$, and 
	\[ C(x,p) = \min(p(\{0\}), p(\{1\})).\]
	Then $C$ is continuous and concave on $X \times \mathcal P(Y)$, but the only (and therefore optimal) coupling $\mu \otimes \nu \in \Pi(\mu,\nu)$ is not $C$-monotone. Indeed,
	\begin{align*}
		2 C(0,\nu) = 1 > 0 = C(0,\delta_0) + C(0,\delta_1).
	\end{align*}
\end{example}

In the classical optimal transport setting $c$-cyclical monotonicity implies optimality already when the cost function $c$ is bounded from below and real valued, see \cite{Be12}. A similar conclusion cannot be drawn in optimal weak transport as $C$-monotonicity does not even imply optimality when $C$ is lower continuous:

\begin{example}\label{ex:2}
	Let $X = [0,1]$, $Y = [0,1]$ and $C(x,p) := p_{ s,\delta_x}([0,1])=p([0,1]\setminus\{x\})$, which is a measurable cost function (c.f.\ Proposition~\ref{prop:measurability of lebesgue decomp}) and lower semicontinuous for fixed $x \in [0,1]$: Given a weakly convergent sequence $p_k \rightharpoonup p \in \mathcal P(Y)$, the Portmanteau theorem yields
	$$\liminf_k C(x,p_k) = \liminf_k p_k([0,1]\setminus\{x\}) \geq p([0,1]\setminus \{x\}) = C(x,p).$$
	The product coupling $\pi:=\lambda \otimes \lambda \in \Pi(\lambda,\lambda)$ where $\lambda$ denotes the uniform distribution on $[0,1]$ is in fact $C$-monotone: Since $\pi_x = \lambda$ ($\lambda$-almost surely), we have for any $x_1,\ldots,x_N \in [0,1]$ and $q_1,\ldots,q_N\in\mathcal P(Y)$ with
	$$\sum_{i=1}^N \pi_{x_i} = N\lambda = \sum_{i=1}^N q_i$$
	that $q_i$ is absolutely continuous to $\lambda$, hence,
	$$\sum_{i=1}^N C(x_i,\pi_{x_i}) = N = \sum_{i=1}^N C(x_i,q_i).$$
	But the unique optimizer { of \eqref{WOT} with cost $C$} is given by $\pi^*(dx,dy) := \lambda(dx)\delta_x(dy)$ and
	\[\int_{[0,1]} C(x,\pi^*_x)\,\lambda(dx) = 0 < 1 = \int_{[0,1]} C(x,\pi_x)\,\lambda(dx).\]
\end{example}

Even when $C$ is given as the integral with respect to some $c\colon X\times Y\rightarrow \R$, we cannot hope that $C$-monotonicity implies optimality and/or $c$-cyclical monotonicity, as the next example shows.

\begin{example}\label{ex:3}
	Let $X=[0,1]$, $Y=[0,1]$ and $C(x,p) = \int_{[0,1]} c(x,y)p(dy)$ with $c(x,y) := \mathbbm{1}_{\{x\}}(y)$. As in the previous example, the product coupling $\pi = \lambda \otimes \lambda$ is $C$-monotone, but not optimal, whereas $\pi^*(dx,dy)= \lambda(dx)\delta_x(dy)$ is optimal and in particular $c$-cyclical monotone.
\end{example}

The failure of $C$-monotonicity to provide optimality in these simple settings (the cost function $C$ is even bounded and lower semicontinuous) is caused by the manner it varies over `$X\times Y$': The variation over $X$ is pointwise (similar to $c$-cyclical monotonicity), whereas over $Y$ variations are taken in a weak sense, i.e., we require that the $Y$-intensities of the two competing sequences {to} agree. Here, $C$-monotonicity is unable to detect the jump from 1 to 0 at $x$, and we could argue that $C$-monotonicity yields optimality of $\pi$ under all couplings $\chi$ in $\Pi(\lambda,\lambda)$ such that $\chi_x \ll \lambda$ for $\lambda$-almost all $x\in [0,1]$. To be able to compare with all competing couplings, more regularity of $C$ in `$Y$-direction' is necessary, e.g., upper semicontinuity as in Theorem~\ref{thm:C-Mon equiv c-cycl} or uniform equicontinuity as in Theorem~\ref{thm:WOT optimality}, but the question remains open precisely how much regularity is required. The authors conjecture that $C$ being upper Carath\'eodory, i.e., $C\colon X\times\mathcal P(Y)\to\R$ is jointly measurable and for fixed $x\in X$ the map $p\mapsto C(x,p)$ is upper semicontinuous, is required for sufficiency of $C$-monotonicity as optimality criterion.

Notably, Example~\ref{ex:2}, when taking $\tilde C := -C$ as cost, provides an upper semicontinuous cost function which is convex (in fact linear) in the following sense: Let $Q\in \mathcal P(\mathcal P([0,1]))$, then
\begin{equation}
	\label{eq:jensen}
	\int_{\mathcal P([0,1])} \tilde C(x,p) Q(dp) = \tilde C\left(x,\int_{\mathcal P([0,1])} p Q(dp)\right).
\end{equation}
Therefore, {$C$ being lower semicontinuous and convex }
is a strictly stronger assumption than the convexity property stated {in \eqref{eq:jensen}} together with measurability, as demanded in Theorem~\ref{thm:WOT monotonicity}.

\section{Appendix}

{ The existence of a concave modulus of continuity was employed in the proof of Theorem \ref{thm:WOT optimality}. We split the argument into a general part, and a part optimized for the setting of the paper. 

\begin{lemma}\label{lem:mod_cont_gnral}
Let $(Y,d)$ be a metric space, $r\geq 1$, and $f\colon X \times Y\rightarrow \R$ be $d$-uniformly continuous uniformly in $x\in X$, i.e.\ such that
	\[	\theta(\delta) := \sup\left\{ |f(x,a) - f(x,b)| \colon x \in X,\, a,b \in Y \text{ s.t.} \, d(a,b)^r \leq \delta \right\},	\]
	vanishes for $\delta \searrow 0$. Assume additionally that
	$$\forall c>0 \colon\, \sup_{\substack{a,b \in Y \colon d(a,b)\geq c,\\ x\in X}} \frac{|f(x,a)-f(x,b)|}{d(a,b)^r}<\infty.$$
	Then there is $\tilde\theta:[0,\infty)\to [0,\infty]$ concave such that $\lim_{\delta\searrow 0}\tilde \theta(\delta)=0$ and
	$$|f(x,a)-f(x,b)|\leq \tilde\theta\left( d(a,b)^r\right).$$	
\end{lemma}

\begin{proof}
Fixing $\epsilon>0$, there is $\delta_\epsilon$ such that $|f(x,a)-f(x,b)|\leq \epsilon$ if $d(a,b)\leq \delta_\epsilon$. If on the other hand $d(a,b)> \delta_\epsilon$ then
\begin{align*}
|f(x,a)-f(x,b)|&\leq d(a,b)^r\, \sup_{\substack{\bar a,\bar b \in Y\colon d(\bar a,\bar b)\geq \delta_\epsilon,\\ \bar x\in X}} \frac{|f(\bar x,\bar a)-f(\bar x,\bar b)|}{d(\bar a,\bar b)^r} =:  d(a,b)^r\, K_\epsilon.
\end{align*}
Hence overall $|f(x,a)-f(x,b)|\leq \epsilon+ d(a,b)^rK_\epsilon$ and in particular
$$\theta(\delta)\leq \epsilon+ \delta K_\epsilon,$$
and  $K_\epsilon<\infty$. Denoting by $\tilde{\theta}$ the concave envelope of $\theta$, i.e.\ the infimum over all affine functions majorizing $\theta$, we conclude $\tilde \theta(\delta)\leq \epsilon+ \delta K_\epsilon,$ and so $\lim_{\delta\to 0}\tilde \theta(\delta)=0$, as $\epsilon$ was arbitrary. Finally remark that $|f(x,a)-f(x,b)|\leq \theta(d(a,b)^r)\leq \tilde \theta(d(a,b)^r) $.
 
\end{proof}

\begin{lemma}\label{lem:mod_cont_wasser}
Let $C\colon X\times \mathcal P_r(Y)\rightarrow \R$ and suppose that $C(x,\cdot)$ is $\mathcal W_r$-uniformly continuous uniformly in $x\in X$, i.e.\ such that
	\[	\theta(\delta) := \sup\left\{ |C(x,p) - C(x,q)| \colon x \in X, (p,q) \in \mathcal P_r(Y)^2 \text{ s.t.} \mathcal W_r(p,q)^r \leq \delta \right\},	\]
	vanishes for $\delta \searrow 0$.
Then there is $\tilde\theta:[0,\infty)\to [0,\infty]$ concave such that $\lim_{\delta\searrow 0}\tilde \theta(\delta)=0$ and
	$$|C(x,p)-C(x,q)|\leq \tilde\theta\left( \mathcal W_r(p,q)^r\right).$$	
\end{lemma}

\begin{proof}
We apply Lemma \ref{lem:mod_cont_gnral}.
It suffices to show that for any $c > 0$
\begin{equation}\label{eq:to show}
	\sup_{p,q \in \mathcal P_r(Y) \colon \mathcal W_r(p,q)^r \geq c}\frac{|C(x,p)-C(x,q)|}{\mathcal W_r(p,q)^r}<\infty.
\end{equation}
To this end choose $0 < \delta \leq \frac{c}{2}$ such that $\theta(\delta)<\infty$.
For any $p,q \in \mathcal P_r(Y)$ with $\mathcal W_r(p,q)^r \geq c$ there is $N \geq 2$ such that
\[
	(N-1)\delta \leq \mathcal W_r(p,q)^r \leq N\delta .
\]
Denoting $[p,q]_\alpha=(1-\alpha)p+\alpha q$, by convexity of optimal transport we have\footnote
{Indeed, if $\beta >\alpha$ we write $[p,q]_\beta := \frac{1-\beta}{1-\alpha} [p,q]_\alpha + \frac{\beta-\alpha}{1-\alpha}q$ and first deduce $\mathcal W_r([p,q]_\alpha,[p,q]_\beta)^r\leq \frac{\beta-\alpha}{1-\alpha}\mathcal W_r([p,q]_\alpha,q)^r $, by convexity of optimal transport with respect to marginals. Iterating the argument we find $\mathcal W_r([p,q]_\alpha,[p,q]_\beta)^r\leq (\beta-\alpha)\mathcal W_r(p,q)^r $.} for all $\alpha,\beta\in[0,1]$
\[
	\mathcal W_r([p,q]_\alpha,[p,q]_\beta)^r\leq |\alpha-\beta|\mathcal W_r(p,q)^r.
\]
Hence,
\begin{align*}
|C(x,p)-C(x,q)|&\leq \sum_{k=1}^N \left|C(x,[p,q]_{(k-1)/N})-C(x,[p,q]_{k/N})\right|\\
&\leq \sum_{k=1}^N \theta \left( \frac{\mathcal W_r(p,q)^r}{N} \right) \leq N \theta (\delta) \leq \frac{\theta(\delta) \mathcal W_r(p,q)^r N}{\delta (N-1)} \leq 2 \mathcal W_r(p,q)^r \frac{\theta(\delta)}{\delta}.
\end{align*}
and we conclude that the left-hand side of \eqref{eq:to show} is bounded by $\frac{2\theta(\delta)}{\delta} < \infty$.
\end{proof}
}

{The following lemma can be viewed as a measurable version of Lusin's theorem.
It is applied in Section \ref{sec:mot} and \ref{sec:wotot} to connect martingale $C$-monontone respectively $C$-monotone sets $\Gamma$ with $(c,\mathcal F_M)$-monotone respectively $c$-cyclically monotone sets $\hat \Gamma$.
Namely this is done in Lemma \ref{lem:c,mart mon equiv C,mart mon} and Proposition \ref{prop:connection}.
}

\begin{lemma}\label{lem:Gammahat}
	Let $\Gamma \subset X \times \mathcal P(Y)$ be analytic, and $c\colon X\times Y \rightarrow \R \cup \{\pm\infty\}$ be Borel measurable.
	Then there exists an analytic set $\hat \Gamma \subset X\times Y$ with the following properties:
	\begin{enumerate}[label =(\roman*)]
	\item \label{it:Gammahat1}
		For any $(x,p) \in \Gamma$ we have that $p$ is concentrated on the fibre $\hat \Gamma_x = \{y \in Y \colon (x,y) \in \hat \Gamma\}$, i.e., $p(\hat\Gamma_x) = 1$.
	\item \label{it:Gammahat2}
		For any $(x,y) \in \hat\Gamma$ we find $(x,p) \in \Gamma$ and a Borel measurable set $K \subset \hat \Gamma_x$ such that
	\begin{enumerate}
		\item $c$ restricted to the fibre $\{x\} \times K$ is continuous,
		\item $y\in \supp(p) \cap K$ and
			\begin{align*}
				\int_{B_\delta(y) \cap K} \frac{c(x,z)}{p(B_\delta(y)\cap K)} p(dz) \rightarrow c(x,y)\quad\text{for }\delta\searrow 0. 
			\end{align*}
	\end{enumerate}
	\end{enumerate}
\end{lemma}


\begin{proof}
	Without loss of generality, we can assume that $c$ is bounded. We follow similar ideas as \cite{Fe81a}. 
	
	{To illustrate the idea, let for a moment $K\subset Y$ be any measurable set, and write $c_x(y) := c(x,y)$.
	Continuity of $c_x \vert_{\{x\} \times K}$ means that for any open set $O \subset \R$ the preimage $c_x^{-1}(O) \cap K$ is open in the trace topology on $Y \cap K$.
	For any radius $\alpha > 0$ and open set $O \subset \R$, we can consider an $\alpha$-doughnut around $c_x^{-1}(O)$, i.e.,
	\[
		\tilde A(\alpha) := \left\{y \in Y \colon \exists z \in c^{-1}_x(O) \text{ s.t.\ } d_Y(y,z) < \alpha \right\} \setminus c_x^{-1}(O).
	\]
	Clearly, $\tilde A(\alpha) \cup c_x^{-1}(O)$ is open and $c_x^{-1}(O) = (\tilde A(\alpha) \cup c_x^{-1}(O)) \cap \tilde A^c(\alpha)$, thus, $c_x^{-1}(O)$ is a relatively open subset of $\tilde A^c(\alpha)$.
	Note that for any $p \in \mathcal P(Y)$, we have by outer regularity that $p(\tilde A(\alpha))$ vanishes with $\alpha$.	

	Next, we generalize the above reasoning adequately to the product space $X \times Y$ and countably many open sets. 
	Denote by $\{U_n\}_{n\in\N}$ a countable basis of the topology on $\R$. 
	In analogy to $\tilde A$, we define for any sequence of radii in $Y$-direction $\alpha = (\alpha_n)_{n \in \N} \in \R_+^\N$, $\alpha_n > 0$ for all $n \in \N$, the $\mathcal B(X\times Y)$-measurable set
	\[
		A(\alpha) := \bigcup_{n\in\N} \left\{(x,y) \in X \times Y \colon \exists (x,z) \in c^{-1}(U_n)\text{ s.t.\ }d_Y(y,z) < \alpha_n \right\} \setminus c^{-1}(U_n).
	\]
	We see likewise that $c^{-1}(U_n) \cap A(\alpha)^c \cap \{x\} \times Y$ is relatively open in $A(\alpha)^c \cap \{x\} \times Y$.
	As $\{U_n\}_{n\in\N}$ forms a basis of the topology on $\R$, we find for any open set $O \subset \R$
	a subset $N \subset \N$ with $\bigcup_{n \in N} U_n = O$, thus,
	\[
		c^{-1}(O) \cap A(\alpha)^c \cap \{x \} \times Y = \bigcup_{n \in N} c^{-1}(U_n)	\cap A(\alpha)^c \cap \{x \} \times Y,
	\]
	is also relatively open in $A(\alpha)^c \cap \{x \} \times Y$, and consequently $c\vert_{A(\alpha)^c \cap \{x\} \times Y}$ is continuous for all $x\in X$.
	
	Let $(x,p,\alpha) \in X\times \mathcal P(Y)\times\R_+^\N$ and $\alpha_n > 0$ for all $n\in\N$, $K := \proj_Y(A(\alpha)^c \cap \{x\} \times Y)$.
	We have shown for any $y \in \supp(p\vert_{K})$ that $\delta \searrow 0$ implies
	\begin{align*}
		\frac{1}{p(K \cap B_\delta(y))} p\vert_{K\cap B_\delta(y)} \to \delta_y\quad\text{in }\mathcal P(Y).
	\end{align*}
	Since $c$ is bounded and continuous on $\{x\}\times K$, we have for $\delta \searrow 0$
	\begin{equation}
		\label{eq:cost continuous}
		\int_{K\cap B_\delta(y)} \frac{c(x,z)}{p(K\cap B_\delta(y))} p(dz) \rightarrow c(x,y).
	\end{equation}

	For any $(x,p) \in X \times \mathcal P(Y)$ we have by outer regularity of $p$ that for any $n \in \N$ and
	$\delta \searrow 0$
	\[
		p\left(\left\{ y \in Y \colon \exists z \in c^{-1}_x(U_n) \text{ s.t.\ }d_Y(y,z) < \delta \right\}\right) \to p\left(c_x^{-1}(U_n)\right).
	\]
	Therefore, for any $\epsilon > 0$ there is $\alpha \in \R_+^\N$ with $\alpha_n > 0, n \in \N$ such that
	\[
		p\left(\left\{ y \in Y \colon \exists z \in c^{-1}_x(U_n) \text{ s.t.\ }d_Y(y,z) < \alpha_n \right\} \setminus c_x^{-1}(U_n)\right) < \frac{\epsilon}{2^n},
	\]
	whence
	\[
		\delta_x \otimes p\left(A(\alpha)\right) \leq \sum_{n \in \N} p\left(\left\{y \in Y \colon \exists z \in c_x^{-1}(U_n) \text{ s.t.\ }d_Y(y,z) < \alpha_n \right\}  \right) < \epsilon.
	\]
	The set
	\[
		M^\epsilon := \left\{ (x,p,\alpha) \in X\times \mathcal P(Y) \times \R_+^\N\colon \delta_x \otimes p(A(\alpha)) < \epsilon\right\}
	\]
	satisfies that $\proj_{X \times \mathcal P(Y)} M^\epsilon = X \times \mathcal P(Y)$.

	We claim that $(x,p,\alpha) \mapsto \delta_x \otimes p\vert_{A(\alpha)^c}$ and $M^\epsilon$ are both Borel measurable:
	Evidently, $\alpha \mapsto A(\alpha)$ meets the requirements of Lemma~\ref{lem:measurable indicator} below, and consequently the function $(x,y,\alpha)\mapsto f_A(x,y,\alpha):=  \mathbbm{1}_{A(\alpha)}(x,y)$ is Borel measurable.
	From here, we deduce Borel measurability of
	\[
		(x,p,\alpha) \mapsto \delta_x \otimes p(\{x\} \times B \setminus A(\alpha)) = \int_B \left|1 -  f_A(x,y,\alpha)\right|\, p(dy),
	\]
	for every $B\in \mathcal B(Y)$, which yields our claim.

	Finally, we are in a position to define $\hat \Gamma$:
	Since $\Gamma$ is analytic and $M^\epsilon$ is Borel measurable, the set $M^\epsilon_\Gamma := M^\epsilon \cap \Gamma \times \mathbb R^\N_+$ is yet again analytic.
	Consider the analytic set
	\begin{align*}
		\Theta^\epsilon :=& \left\{(x,p,\alpha,y) \in X\times \mathcal P(Y) \times \R^N_+ \times Y  \colon \exists (x,p,\alpha) \in M^\epsilon_\Gamma \text{ s.t.\ }y \in \supp(p), \mathbbm 1_{A(\alpha)}(x,y) = 0\right\}
	\\	=& M^\epsilon_\Gamma \cap \left\{(x,p,\alpha,y) \in X \times \mathcal P(Y) \times \R^\N_+ \times Y \colon y \in \supp(p), f_A(x,y,\alpha) = 0 \right\},
	\end{align*}
	where we used that $Y$ admits a countable basis $(O_k)_{k\in\N}$ of its topology, whereby
	\[
		y \in \supp(p) \iff \forall k \in \N\colon \text{either }p(O_k) > 0 \text{ or }y\notin O_k.	
	\]
	Since $\Theta^\epsilon$ is analytic, its $X\times Y$-projection is analytic too. Thus,
	\[
		\hat \Gamma := \bigcup_{k \in \N}\proj_{X \times Y} \Theta^{\frac{1}{k}}
	\]
	is analytic.
	Recall that for any $(x,p) \in \Gamma$ and $k \in \N$ there is $\alpha^k \in (0,\infty)^\N$ with $\delta_x \otimes p(A(\alpha^k)) < \frac{1}{k}$.
	As for every $k \in \N$ we have
	\begin{align*}
		p(\hat \Gamma_x) &= \delta_x \otimes p(\hat \Gamma \cap X \times \supp(p)) \geq \delta_x \otimes p(\{x \} \times \{y \in Y \colon f_A(x,y,\alpha^k) = 0\}) 
	\\	&= 1 - \delta_x \otimes p(A(\alpha^k)) > 1 - \frac{1}{k},
	\end{align*}
	we derive that $\hat \Gamma$ satisfies item \ref{it:Gammahat1}.

	On the other hand, if $(x,y) \in \hat \Gamma$, then there is $k \in \N$ and $(x,p,\alpha) \in M^\epsilon_\Gamma$ with $(x,p,\alpha,y) \in \Theta^{\frac{1}{k}}$.
	Set $K := \proj_Y(A(\alpha)^c \cap \{x\} \times Y)\cap \supp(p)$ and note that
	$\{x \} \times \{p\} \times \{\alpha\} \times K \subset \Theta^\epsilon$, that means $K \subset \hat \Gamma_x$.
	Recall \eqref{eq:cost continuous}. By the paragraph preceding \eqref{eq:cost continuous}, we have continuity of $c\vert_{\{x\} \times K}$, and conclude that item \ref{it:Gammahat2} is satisfied.	}
\end{proof}

\begin{lemma}\label{lem:measurable indicator}
	Let $A\colon \R^\N_+ \to \mathcal B(X\times Y)$ be given s.t.
	\begin{enumerate}[label = (\alph*)]
	\item for any\footnote{In this part we write $\alpha \leq  \beta\iff \alpha_i\leq \beta_i$ for all $i$.} $\alpha,\beta\in\R^\N_+$, $\alpha \leq  \beta$ we have $A(\alpha) \subset A(\beta)$,
		\item for any $\{\alpha^k\}_{k \in \N} \subset \R^\N_+$ { with $\alpha^k_n\nearrow \alpha_n, n \in \N$ and $\alpha=(\alpha_n)_{n \in \N} \in \R^\N_+$}, we have $A(\alpha^k) \nearrow A(\alpha)$.
	\end{enumerate}
	Then the function
	\begin{align*}
		f_A\colon X\times Y\times \R^\N_+ \to \R\colon (x,y,\alpha) \mapsto \mathbbm{1}_{A(\alpha)}(x,y)
	\end{align*}
	is Borel measurable.
\end{lemma}

\begin{proof}
	The function $f_A$ is the indicator function of the set
	$$\mathcal A := \left\{(x,y,\alpha) \in X\times Y\times \R^\N_+\colon (x,y) \in A(\alpha) \right\} = \bigcup_{\alpha \in \R^\N} A(\alpha) \times \{\alpha\}.$$
	We show that $\mathcal A$ coincides with the Borel measurable set $\mathcal A'$,
	$$\mathcal A' := \bigcup_{q \in \mathcal Q} A(q) \times \bigotimes_{n\in\N} [q_n,\infty),$$
	where $\mathcal Q := \{\alpha \in \Q^\N_+\colon \exists n\in\N\text{ s.t. }\alpha_k = 0~\forall k\geq n\}$. From property $(a)$ of $A$ we deduce $\mathcal A \supset \mathcal A'$. Now let $(x,y,\alpha) \in \mathcal A$, then we find a sequence $\{q^k\}_{k\in\N}\subset \mathcal Q$ with $q^k \nearrow \alpha$, and by $(b)$, there exists an index {$k_0\in\N$} such that $(x,y) \in A(q^k)$ for all $k\geq k_0$. Thus,
	$$(x,y,\alpha) \in A(q^{k_0}) \times \bigotimes_{n\in\N} [q^{k_0}_n,\infty) \subset \mathcal A',$$
	and $\mathcal A \subset \mathcal A'$. Therefore $\mathcal A$ is a Borel measurable set and $f_A$ a Borel measurable function.
\end{proof}

We conclude this section by showing the following measurable variant of the Lebesgue decomposition theorem, quoted in Example \ref{ex:2}.

\begin{proposition}\label{prop:measurability of lebesgue decomp}
	Let $\mathcal M_+(X)$ the space of all finite measures on $X$ be equipped with the topology of weak convergence of measures. Then the map
	\begin{align*}
		T\colon \mathcal M_+(X)\times \mathcal M_+(X) \rightarrow \mathcal M_+(X)\times \mathcal M_+(X),\\ (p,q)\mapsto (q_{ac,p},q_{s,p}),
	\end{align*}
	where $T(p,q)$ is the unique Lebesgue decomposition of $q$ w.r.t.\ $p$, i.e.,
	\begin{align*}
	q_{ac,p} + q_{s,p} = q,\quad q_{ac,p} \ll p,\quad q_{s,p}\perp p,
	\end{align*}
	is measurable.
\end{proposition}

\begin{proof}
Essentially we are defining a family of measurable functions: for any $\delta>0$ and $A\in \mathcal B(X)$ let
\begin{align}\label{eq:def Ffdelta}
\begin{split}
F_{\delta,A} \colon \mathcal M_+(X)\times \mathcal M_+(X) \times \mathcal B(X) \rightarrow \R,\\
(p,q,B)\mapsto \begin{cases} q(A\setminus B)& p(B)<\delta,\\ q(A)&\text{else.}\end{cases}
\end{split}
\end{align}
In turn, this allows us to define $T(p,q)(A)$ as a the limit of countable infima of measurable functions. Since $X$ is Polish there exists a countable family $\mathcal O$ of open sets on $X$ such that $\sigma(\mathcal O) = \mathcal B(X)$. Denote by $\mathcal R(\mathcal O)$ the set-theoretic ring generated by $\mathcal O$, which is again countable. Then, for any $\epsilon>0$, $A\in \mathcal B(X)$ and $p\in\mathcal P(X)$ there is a set $A_\epsilon \in \mathcal R(\mathcal O)$ such that
\begin{equation}
	\label{eq:ring approx}
	p(A\setminus A_\epsilon \cup A_\epsilon\setminus A) < \epsilon.
\end{equation}
Using the classical Lebesgue decomposition theorem, for any pair $(p,q) \in
\mathcal M_+(X) \times \mathcal M_+(X)$ there is  a set $\tilde B \subset X$ with $p(\tilde B) = 0$ and 
\[
	q_{ac,p}(A) = q(A \setminus \tilde B) \quad \forall A \in \mathcal B(X).
\]
By \cite[Proposition 4.5.3]{Bo00} we have
\[
	\lim_{\delta \searrow 0} \sup_{B \in \mathcal B(X) \colon p(B) < \delta} q_{ac,p}(B) = 0,
\]
whence,
\begin{align*}
	q_{ac,p}(A) &= \lim_{\delta \searrow 0} \inf_{B \in \mathcal B(X) \colon p(B) < \delta} q_{ac,p}(A\setminus B)
\\	&= \lim_{\delta \searrow 0} \inf_{B \in \mathcal B(X) \colon p(B) < \delta} q_{ac,p}(A\setminus(\tilde B \cup B)) + q_{s,p}(A\setminus(\tilde B \cup B))
\\	&= \lim_{\delta \searrow 0} \inf_{B \in \mathcal B(X) \colon p(B) < \delta} q(A \setminus B)
\\	&= \lim_{\delta \searrow 0} \inf_{B \in \mathcal R(\mathcal O) \colon p(B) < \delta} q(A \setminus B),
\end{align*}
where we used for the last equality the approximation property \eqref{eq:ring approx}.

Thus $F_A\colon \mathcal M_+(X)\times \mathcal M_+(X)\rightarrow \R$ defined by
\begin{align}\label{eq:measurable}
F_A(p,q) := q_{ac,p}(A) = \inf_{\delta\searrow 0}\inf_{\{B\in\mathcal B(X)\colon p(B) { < \delta} \}} q(A\setminus B) = \lim_{k \to \infty} \inf_{B \in \mathcal R(\mathcal O)} F_{\frac{1}{k},A}(p,q,B)
\end{align}
is Borel measurable, {and as $A\in\mathcal B(X)$ was arbitrary we conclude
that $T$ is measurable.}
\end{proof}

\bibliography{joint_biblio}{}
\bibliographystyle{abbrv}
\end{document}